\documentclass{amsart}
\usepackage[T1]{fontenc}
\usepackage[english]{babel}
\usepackage{amsmath, amssymb, amsthm}
\usepackage{amsmath, mathtools}  
\usepackage{mathrsfs}
\usepackage{geometry}
\usepackage{graphicx}
\usepackage{fancyhdr} 
\usepackage{tikz}
\usetikzlibrary{decorations.markings,arrows.meta}
\usepackage{amscd}
\usepackage{epsf}
\usepackage{booktabs}
\usepackage{array}
\usepackage{multirow}
\usepackage{color}
\usepackage{footmisc}
\usepackage{lipsum}
\usepackage{latexsym}
\usepackage{verbatim}
\usepackage{bbm}
\usepackage{setspace}
\usepackage{hyperref}
\usepackage{cleveref}
\usepackage{etoolbox}
\usepackage[normalem]{ulem} 
\usepackage{thmtools} 
\usepackage{enumitem}
\usepackage{anyfontsize}

\renewcommand{\bar}[1]{\overline{#1}}

\setlist[enumerate,1]{label=\textbf{\arabic*.}, ref=\arabic*, leftmargin=*, itemsep=0.5ex, topsep=0.5ex}

\let\OLDthebibliography\thebibliography
\renewcommand\thebibliography[1]{
  \OLDthebibliography{#1}
  \setlength{\parskip}{2pt}
  \setlength{\itemsep}{0pt plus 0.3ex}
}

\makeatletter
\renewenvironment{proof}[1][\proofname]{%
  \par\pushQED{\qed}%
  \normalfont\topsep6\p@\@plus6\p@\relax
  \trivlist
  \item[\hskip\labelsep
        \textbf{\textup{#1}}\@addpunct{.}]\ignorespaces
}{%
  \popQED\endtrivlist\@endpefalse
}
\makeatother
\newcommand{\mysubjclass}{\textup{2020} \textbf{Mathematics Subject Classification:} 53E30, 53C26}
\newcommand{\mykeywords}{\textbf{Keywords:} hypercomplex geometry, hyperHermitian geometry, HKT, Hermitian curvature flow, conformal geometry}

\fancypagestyle{firstpagefooter}{
  \fancyhf{} 
  \fancyfoot[L]{\scriptsize
    \mysubjclass \\
    \mykeywords \\
    }
}

\hypersetup{
    colorlinks=true,
    linkcolor=blue,    
    citecolor=red,     
    urlcolor=blue,     
    pdftitle={Hermitian curvature flow and HKT structures},
    pdfauthor={B. Brienza and J. Streets}
}

\newtheorem{thm}{Theorem}[section]
\newtheorem{prop}[thm]{Proposition}
\newtheorem{lem}[thm]{Lemma}

\newtheorem{conj}[thm]{Conjecture}

\newtheorem{defn}[thm]{Definition}

\theoremstyle{definition}
\newtheorem{rmk}[thm]{Remark}

\numberwithin{equation}{section}

\newcommand{\brs}[1]{\left|#1\right|}
\newcommand{\til}[1]{\widetilde{#1}}
\newcommand{\del}{\partial}
\newcommand{\dt}{\frac{\del}{\del t}}

\newcommand{\N}{\mathbb{N}}

\newcommand{\R}{\mathbb{R}}

\newcommand{\SU}{\mathrm{SU}}
\newcommand{\SO}{\mathrm{SO}}
\newcommand{\Sp}{\mathrm{Sp}}

\newcommand{\Id}{\mathrm{Id}}

\newcommand{\tr}{\mathrm{tr}}

\renewcommand{\epsilon}{\varepsilon}
\renewcommand{\phi}{\varphi}

\renewcommand{\i}{\sqrt{-1}}
\renewcommand{\N}{\nabla}

\newcommand{\gl}{\lambda}
\newcommand{\dd}{\,d}

\usetikzlibrary{shapes.misc}
\tikzset{cross/.style={cross out, draw, 
         minimum size=2*(#1-\pgflinewidth), 
         inner sep=0pt, outer sep=0pt},
         cross/.default={3.3pt}}

\def\eqref#1{(\ref{#1})}
\title{Hermitian curvature flow and HKT geometry}

\author{Beatrice Brienza}
\address{Beatrice Brienza: IMPA\\
         Estrada Dona Castorina, 110 \\
         CEP 22460-320 \\
         Rio de Janeiro, RJ -- Brasil}
\email{brienza.beatrice@impa.br}

\author{Jeffrey Streets}
\address{Jeffrey Streets: Rowland Hall \\
         University of California, Irvine \\
         Irvine, CA 92617, USA}
\email{jstreets@uci.edu}

\date{\today}

\begin{document}

\begin{abstract} We identify a Hermitian curvature flow which preserves HKT geometry, and whose fixed points are HKT-Einstein metrics, equivalent to a flow suggested by Verbitsky in the context of the quaternionic Monge-Amp\`ere equation.  We exhibit a fundamental regularity obstruction and a monotonicity formula for the Chern scalar curvature.  We formulate a maximal existence time conjecture for this flow, and give a conditional resolution.  We establish the existence conjecture in dimension four.  We show the existence of a divergent sequence of HKT-Einstein metrics on quaternionic Hopf surfaces. These are the first non-homogeneous examples in the literature, and indicate the delicacy of the convergence question.  Finally we classify which strong HKT structures arising from bi-invariant metrics on Lie groups are also HKT-Einstein.
\end{abstract}

\maketitle

\thispagestyle{firstpagefooter}

\section{Introduction}

The study of Hermitian metrics on complex manifolds beyond the K\"ahler setting is a topic of great interest, motivated in part by its connections with string theory through the work of Hull and Strominger \cite{hull1986superstring,StromingerSST}.  A case of particular interest arises from generalizing hyperK\"ahler geometry, leading to the notion of \emph{hyperK\"ahler with torsion} (HKT) metrics, introduced by Howe and Papadopoulos \cite{HowePap}. An HKT manifold is a hyperhermitian manifold admitting a metric connection with skew-symmetric torsion whose holonomy is contained in $\mathrm{Sp}(n)$. From the viewpoint of physics, HKT geometry was introduced in \cite{HowePap} as the natural geometric framework for supersymmetric sigma models with Wess--Zumino terms \cite{grover2009,gutowski2000,gutowski2011}. Since then, HKT geometry has attracted significant attention in differential geometry, see, for instance, \cite{grantcharov2000,alesker2006,fusi2026,Joy,Ver02} and the references therein.  Of particular interest are \emph{strong HKT metrics}, i.e., those with closed skew-symmetric torsion.  We note that in early works (i.e. \cite{HowePap}) the notion of HKT metrics included this closure condition. Although HKT structures are abundant, compact strong HKT manifolds with non-vanishing torsion remain rare: all known examples are locally products of hyper-K\"ahler manifolds with group manifolds. Nevertheless, in low dimensions the geometry of strong HKT manifolds is now nearly understood, thanks to recent advances \cite{brienza2026, papadopoulos2025}.  More recently, \cite{fusi2026} introduced the notion of an \emph{HKT-Einstein} metric, a natural analogue of K\"ahler-Einstein metrics in the HKT setting. These metrics were further investigated in \cite{BedulliMarcocci2026, brienza2026}.

\medskip

In view of these special metrics in HKT geometry, note that in recent years various geometric flows have been employed to construct canonical Hermitian non-K\"ahler metrics (cf. survey articles \cite{phong2024geometric, st-geom, tosatti2021chern}).  In \cite{HCF} a general family of parabolic flows of Hermitian metrics was introduced, with the view that different flows would be well-adapted to the vast variety of natural structures in Hermitian geometry.  In the following years several different flows were shown to preserve certain natural conditions and saw geometric applications \cite{fei2020unification,JFS,LeeHCF, PCF, GKRF, UstinovskiyHCF}.  The present work derives a Hermitian curvature flow as in \cite{HCF} which preserves the HKT condition.  On a hypercomplex background $(M, I, J, K)$ consider the Hermitian curvature flow
\begin{align} \label{f:HKTflow}
    \frac{\partial}{\partial t} \omega =&\ - \Phi(g) = - S + Q,
\end{align}
where $Q$ is a quadratic expression built from the Chern and Bismut torsions (cf. \eqref{eqn:Q3}) and $\omega$ denotes the fundamental form of $(g_t,I)$.  Our first main result collects the formal properties of this flow which hold for all Hermitian curvature flows, such as short time existence and uniqueness, and a long-time existence obstruction.  Furthermore, the crucial new point is that it preserves HKT geometry, in which case we show it is the same as that appeared in \cite{BGV}, and we express it as Ricci flow coupled to torsion.

\begin{thm} \label{t:mainthm1} Let $(M^{4n}, I, J, K)$ be a compact hypercomplex manifold.  Given $g_0$ hyperHermitian:
\begin{enumerate}[label={(\arabic*)}]
    \item There exists a unique maximal smooth solution to (\ref{f:HKTflow}) on $[0, \tau)$ with initial condition $g_0$ for $\tau \leq \infty$.
    \item If $\tau < \infty$ then 
    \begin{align*}
        \limsup_{t \to \tau}  \left\{ \brs{R^C}_{g_t}, \brs{T}_{g_t}^2, \brs{\N^C T}_{g_t} \right\} = \infty,
    \end{align*}
    where $R^C$ is the curvature of the Chern connection and $T$ is its torsion.
    \item If $g_0$ is HKT then:
    \begin{enumerate} 
        \item $g_t$ is HKT for all $0 \leq t < \tau$.
        \item The associated HKT forms $\Omega_t$ satisfy
    \begin{align} \label{f:Omegaflow}
        \dt \Omega_t = \partial_J \alpha,
    \end{align}
    where $\alpha$ is the Obata connection $1$-form on the canonical bundle.
    \item The associated Riemannian metrics satisfy
    \begin{align} \label{f:HKTasGRF}
        \dt g = - \mathrm{Rc} +\tfrac14 H^2 - \tfrac12 \mathcal{L}_{\theta^\sharp}g - \tfrac 12 \hat Q^2,
    \end{align}
    where $\hat Q^2$ is defined in \eqref{eqn:Q_hat}.
    \end{enumerate}     
\end{enumerate}
\end{thm}
\noindent 
Thus we have a well-posed flow which preserves HKT geometry.  In case the initial data is HKT we refer to this as a solution to \emph{HKT Ricci flow}.  We note that the flow (\ref{f:Omegaflow}) was originally defined in \cite{BGV} as a tool for understanding the quaternionic Monge-Amp\`ere equation and HKT-Einstein metrics.  Curiously, the evolution equation (\ref{f:HKTasGRF}) differs from the the metric evolution for generalized Ricci flow only by the term $\hat{Q}^2$, suggesting that it may inherit some similar properties.

A fundamental point in the study of Ricci flow is the scalar curvature montonicity formula.  We show a monotonicity formula along HKT Ricci flow for the Chern scalar curvature, which leads to a priori lower bounds.  In the statement below $\square$ is the time-dependent Chern heat operator.

\begin{thm} \label{t:scalarthm} (cf. Theorem \ref{t:scalarthm_bulk})
 Given $g_t$ a solution to the HKT Ricci flow,
\begin{align*}
\square s= \tfrac 12 |\Phi(g)|^2.
\end{align*}
Assuming $M$ is compact, 
\begin{enumerate} [label={(\arabic*)}]
    \item $\inf_{M \times \{t\}} s \geq \inf_{M \times \{0\}} s$,
    \item $\inf_{M \times \{t\}} s \geq - \tfrac{2n}{t}$,
    \item If $\inf_{M \times \{0\}} s = \sigma > 0$, then the maximal smooth existence time of the flow is $T \leq \tfrac{2n}{\sigma}$.
\end{enumerate}
\end{thm}
The HKT Ricci flow is closely linked to the theory of the parabolic quaternionic Monge-Amp\`ere equation, allowing for more definitive results.  To begin we formulate a maximal existence time conjecture in analogy K\"ahler-Ricci flow criteria established by Tian-Zhang \cite{TianZhang}.  A related conjecture for hyperHermitian metrics flowing by equation (\ref{f:Omegaflow}) appears in \cite[Conjecture 5.1]{BGV}. In the HKT setting, there is a natural notion of positive cone in analogy with the K\"ahler cone, and the conjecture predicts smooth existence of the HKT flow as long as the class remains in this cone:

\begin{conj} \label{c:coneconjintro} (cf. Conjecture \ref{c:coneconj}) Let $(M^{4n}, I, J, K, g_0)$ be a compact HKT manifold.  Let
\begin{align*}
    \tau^* := \sup\ \{ t \geq 0\ |\ [\Omega_0]_{qBC} - t c_1^{qBC} \in \mathcal P \}.
\end{align*}
The maximal smooth solution to (\ref{f:HKTflow}) exists on $[0, \tau^*)$.
\end{conj}
\noindent Global existence and convergence of the flow on  hyperK\"ahler backgrounds was shown in \cite{BGV}, in particular resolving Conjecture \ref{c:coneconjintro} in this case. We prove two conditional resolutions relying on further a priori assumptions.  Case (1) of the following theorem is implicit in \cite{BGV}:
\begin{thm} \label{t:coneconjthm} (cf. Theorem \ref{t:coneconjthmbulk}) Let $(M^{4n}, I, J, K, g_0)$ be a compact HKT manifold.  Suppose the solution to HKT Ricci flow with this initial data exists on $[0, \tau)$ where $\tau < \tau^*$ and either
\begin{enumerate} [label={(\arabic*)}]
    \item $\sup_{M \times [0, \tau)} \tr_{g_{0}}g < \infty$,
    \item $\sup_{M \times [0, \tau)} \brs{T} < \infty$.
\end{enumerate}
Then the flow extends smoothly past time $\tau$.
\end{thm}

Going further, we show that Conjecture \ref{c:coneconjintro}holds in quaternionic dimension $n = 1$.  We state the result instead in terms of the \emph{normalized HKT Ricci flow}.  Noting that HKT-Einstein metrics can only occur with nonnegative Einstein constant (cf. \S \ref{ss:HKTE}), which can then be normalized to $1$ if positive.  Thus we define
\begin{align} \label{eq:normalizedHKTflow}
\dt \omega = - \Phi(g) + \omega.
\end{align}
By standard arguments it follows that this flow also preserves the HKT condition, and differs from HKT Ricci flow by a spacetime rescaling.

\begin{thm} \label{t:4dthm} Fix $(M^4, g, I, J, K)$ a compact HKT manifold.  Then either:
\begin{enumerate} [label={(\arabic*)}]
    \item $(M, I)$ is hyperK\"ahler, and the HKT flow with initial condition $g$ exists on $[0, \infty)$ and converges to the unique hyperK\"ahler metric $g_{\mathrm{HK}}$ conformal to $g$ with $[\Omega_{\mathrm{HK}}] = [\Omega]$.
    \item $(M, I)$ is a quaternionic Hopf surface, and the normalized HKT flow with initial condition $g$ exists on $[0,\infty)$.
\end{enumerate}
\end{thm}
\noindent The hyperK\"ahler background case follows from \cite{BGV} but we provide the elementary proof in dimension four for convenience.  By a standard spacetime rescaling argument the global existence of the normalized flow on Hopf surfaces also implies Conjecture \ref{c:coneconjintro} in this case.  We show that in dimension four the flow reduces to a conformal flow, in fact the Chern-Yamabe flow \cite{Calamai,LM}.  That the background manifold is either hyperK\"ahler or Hopf follows from the Boyer classification of hyperHermitian four-manifolds \cite{Bo}.  On the Hopf surface we show time-dependent estimates for the normalized flow against the Hopf-Boothby background.  

The convergence of the flow in the case of the Hopf surface seems to be a delicate question.  Initially it seems natural to expect that the standard Hopf metric is the unique HKT-Einstein metric with $\gl > 0$ in dimension four, and hence one should expect convergence of the flow to this metric.  In the next theorem we exhibit infinitely many HKT-Einstein metrics on the Hopf surface which blow up, destroying these expectations:

\begin{thm} \label{t:4dnonuniquenessinto} (cf. Theorem \ref{t:4dnonuniquenessbulk}) Consider the standard quaternionic Hopf surface $(S^1 \times \SU(2), I,J,K, g_{\mathrm{Hopf}})$. There exists a sequence $\{u_k\} \in C^{\infty}(S^1 \times \SU(2))$ such that
\begin{enumerate} [label={(\arabic*)}]
    \item $\Phi(e^{u_k} g_{\mathrm{Hopf}}) = e^{u_k} g_{\mathrm{Hopf}}$,
    \item $\lim_{k \to \infty} \brs{u_k}_{C^0} = \infty$.
\end{enumerate}
\end{thm}

\noindent The manifolds $S^n \times S^1$ are a classic source of nonuniqueness for the Yamabe problem (cf. \cite{de2012bifurcation,kobayashi1985conformally,schoen2006variational}).  While similar in spirit, here we do not have a variational structure to work with and rather rely on a brute force ODE analysis of $S^1 \times \SO(3)$ invariant metrics. To the best of our knowledge, the HKT–Einstein metrics ${e^{u_k}g_{\mathrm{Hopf}}}$ constructed in Theorem \ref{t:4dnonuniquenessinto} provide the first examples in the literature of nonhomogeneous HKT–Einstein structures.

We also study strong HKT–Einstein metrics on compact Lie groups. In real dimension four, strong HKT–Einstein structures have been completely classified: when $\lambda\neq 0$, the unique example is the standard Hopf surface. In higher dimensions, however, the general theory remains largely unexplored. As a step toward its development, we obtain a complete classification of strong HKT–Einstein metrics on compact Lie groups.
\begin{thm} (cf. Theorem \ref{t:sHKTE})
Let $G$ be a non-trivial compact semisimple Lie group and let $(\mathbb{T}^\ell \times G, I,J,K,g)$ be any strong HKT manifold constructed as in Section \ref{s:strongHKTE}. Then $(I,J,K,g)$ is HKT-Einstein if and only if
\[
\mathbb{T}^\ell \times G \cong \bigl( S^1 \times \mathrm{SU}(2) \bigr) \times \dots \times \bigl( S^1 \times \mathrm{SU}(2) \bigr) \times \mathrm{SU}(3) \times \dots \times \mathrm{SU}(3).
\] 
\end{thm}
The proof of the theorem is algebraic and relies on the observation that the HKT–Einstein condition can be recast as an equation expressed entirely in terms of the Joyce decomposition associated with the compact semisimple Lie group $G$ (see Equation~\ref{eqn:HKT_Einstein}). We first treat the case in which $G$ is simple through a case-by-case analysis and then extend the argument to arbitrary compact semisimple Lie groups. 

\medskip

\noindent \textbf{Acknowledgments:} The authors thank Mehdi Lejmi and Luigi Vezzoni for helpful comments on an earlier draft of this paper.  We further thank Daniele Angella who pointed out the construction behind the metrics of Theorem \ref{t:4dnonuniquenessinto}, which was suggested by ChatGPT.  Through further consultation with ChatGPT we developed a complete proof.  The authors take full responsibility for the contents of this paper.

\section{HKT geometry} \label{conv}

To set the stage, we first recall the fundamental aspects of hypercomplex and hyperhermitian geometry, referring the reader to \cite{HowePap, grantcharov2000} for a more detailed exposition. With our conventions established, we then proceed to define the key notions of HKT-Einstein geometry and the first quaternionic Bott-Chern class.

\subsection{Conventions}

Let $(M^{2n}, I, g)$ be a Hermitian manifold. We denote by $\nabla$ the Levi‑Civita connection of $g$.

\begin{enumerate} [label={(\arabic*)}]
    \item The fundamental $2$-form is $\omega(X,Y)=g(X,IY)$.
    \item The complex structure $I$ acts on $k$-forms by
    \[
    (I\alpha)(X_1,\dots,X_k)=(-1)^k\alpha(IX_1,\dots,IX_k).
    \]
    \item The \emph{Bismut connection} $\nabla^B$ is the unique Hermitian connection with totally skew‑symmetric torsion $H$. Its torsion is given by $H = d^c\omega = I d\omega$, and the connection is
    \[
    g(\nabla^B_X Y, Z)=g(\nabla_X Y, Z)+\tfrac{1}{2}H(X,Y,Z).
    \]
    \item The curvature tensor $R^B$ of $\nabla^B$  lies in $\Omega^2(M) \otimes \Omega^{1,1}(M)$, i.e.,  $R^B(X,Y,Z,W)=-R^B(Y,X,Z,W)=R^B(X,Y,IZ,IW)$.
    \item The first and second Bismut Ricci curvatures are
    \[
    \rho^B(X,Y)=\tfrac12\sum_{i=1}^{2n} R^B(X,Y,e_i,Ie_i),\qquad
    S^B(X,Y)=\tfrac12\sum_{i=1}^{2n} R^B(e_i,Ie_i,X,Y),
    \]
    for any orthonormal frame $\{e_i\}$.
    \item The \emph{Lee form} $\theta$ is defined by $d\omega^{n-1}=\theta\wedge\omega^{n-1}$. It can be expressed via the torsion as
    \[
    \theta(X)=-\tfrac12\sum_{i=1}^{2n} H(IX,e_i,Ie_i).
    \]
    The Hermitian metric $(g,I)$ is Gauduchon if $d^*\theta=0$.
    \item The \emph{Chern connection} $\nabla^C$ is the unique Hermitian connection whose torsion $T$ satisfies $T(IX,Y)=T(X,IY)$. Its torsion is
    \begin{equation} \label{eqn:chern_torsion}
      T(X,Y,Z)=-\tfrac12\bigl(H(X,IY,IZ)+H(IX,Y,IZ)\bigr).  
    \end{equation}
 
    \item The curvature $R^C$ of $\nabla^C$ lies in $\bigwedge^{1,1}T^*M \otimes \bigwedge^{1,1}T^*M$.
    \item The first and second Chern Ricci curvatures are
    \[
    \rho(X,Y)=\tfrac12\sum_{i=1}^{2n} R^C(X,Y,e_i,Ie_i),\qquad
    S(X,Y)=\tfrac12\sum_{i=1}^{2n} R^C(e_i,Ie_i,X,Y).
    \]
    \item Below we will use the tensor
\[
\lambda^{\omega} (X,Y)
:=
\sum_{i=1}^{2n} dH(X,Y,e_i,Ie_i)=\sum_{i=1}^{2n} dH(e_i,Ie_i,X,Y).
\]
\end{enumerate}

Also we record a lemma on the Lee form coming from \cite{Alexandrov}, matched to our conventions.
\begin{lem} \label{l:trdItheta}
Let $(g,I)$ be a Hermitian structure. Then
\[
 \sum_{i=1}^{2n} dI\theta (Ie_i,e_i)= 2 d^*\theta+ 2 |\theta|^2.
\]
\end{lem}
\begin{proof}
The proof is a computation:
\begin{align*}
\sum_{i=1}^{2n} dI\theta (Ie_i,e_i)&= \sum_{i=1}^{2n} Ie_i(I\theta(e_i))- e_i(I\theta(Ie_i))-I\theta([Ie_i,e_i])\\
&=  \sum_{i=1}^{2n}  -2\,  e_i (\theta (e_i) )+2 \theta(\nabla{e_{i}} e_i )+I\theta(H(Ie_i,e_i))\\
&= 2 d^*\theta - H (e_i, Ie_i, I \theta^\sharp)\\
& =2 \left( d^*\theta+  |\theta|^2 \right).
\end{align*}
\end{proof}

\subsection{Hypercomplex manifolds}

A hypercomplex manifold $M$ is a smooth $4n$-dimensional manifold equipped with a triple of complex structures $I$, $J$, $K$ satisfying the unitary quaternionic relations:
\[
IJ = -JI = K,\qquad JK = -KJ = I,\qquad KI = -IK = J.
\]
In the following, we will refer to $n$ as the \emph{quaternionic dimension of} $M$.
Note that on a hypercomplex manifold, there actually exists a whole $2$-sphere of complex structures:
\[
S^2 = \{ aI + bJ + cK \mid a^2 + b^2 + c^2 = 1 \},
\]
and in the following let $L$ denote a generic element of $S^2$.

\par
\medskip

Given a hypercomplex manifold $(M,I,J,K)$, there exists a unique torsion‑free connection, called the \emph{Obata connection} and denoted by $\nabla^{\mathrm{Ob}}$, satisfying $\nabla^{\mathrm{Ob}}L=0$ for every $L\in S^2$. By the holonomy principle, the holonomy group of $\nabla^{\mathrm{Ob}}$ is contained in $\mathrm{GL}(n,\mathbb{H})$.  The possible irreducible holonomy groups of torsion-free connections contained in $\mathrm{GL}(n,\mathbb{H})$ are severely restricted. Indeed, among the groups appearing in the classification of \cite{MS}, only
\[
\mathrm{Sp}(n),\qquad \mathrm{SL}(n,\mathbb{H}),\qquad \mathrm{GL}(n,\mathbb{H})
\]
occur, where $\mathrm{SL}(n,\mathbb{H})$ denotes the commutator subgroup of $\mathrm{GL}(n,\mathbb{H})$.  The geometric significance of the first two cases is well understood. If $(M,I,J,K)$ admits a hyperk\"ahler metric, then the Obata connection agrees with the Levi--Civita connection of that metric, and consequently its holonomy is contained in $\mathrm{Sp}(n)$. On the other hand, the holonomy of $\nabla^{\mathrm{Ob}}$ is contained in $\mathrm{SL}(n,\mathbb{H})$ if and only if the canonical bundle of $(M,I)$ admits a nonzero section that is parallel with respect to $\nabla^{\mathrm{Ob}}$.

\subsection{Differential operators}

Let $d=\partial+\bar\partial$ denote the decomposition of the exterior differential with respect to the complex structure $I$. Following \cite{Ver02}, we introduce the operator
\[
\partial_J:=(-1)^k J\bar\partial J
\]
acting on $k$-forms. This operator plays the role, in the hypercomplex setting, of the twisted differential $d_I^c$ in complex geometry. It is straightforward to verify that
\[
\partial_J^2=0,
\qquad
\partial\,\partial_J+\partial_J\,\partial=0.
\]

The operators $\partial$ and $\partial_J$ give rise to the quaternionic Bott--Chern cohomology introduced in \cite{GLV},
\[
H^{\bullet,\bullet}_{qBC}(M)
=
\frac{\ker\partial\cap\ker\partial_J}
     {\operatorname{Im}(\partial\partial_J)}.
\]
We denote the corresponding cohomology class of a form by $[-]_{qBC}$. When $M$ is compact, the spaces $H^{\bullet,\bullet}_{qBC}(M)$ are finite-dimensional \cite[Theorem 5.1]{GLV}.

\subsection{Hyperhermitian metrics}

A Riemannian metric $g$ that is Hermitian with respect to $I$, $J$, and $K$ (and hence with respect to every $L\in S^2$) is called a \emph{hyperHermitian metric}. 
Such a hyperHermitian metric defines K\"ahler forms with respect to any of the compatible complex structures 
\[
\omega_L(X,Y):=g(X,LY), \qquad L\in S^2.
\]
Throughout the paper we shall mainly work with the complex structure $I$ and the corresponding K\"ahler form
$\omega_I$. Whenever it is written $\omega$ without the corresponding complex structure, we always mean that it is with respect to $I$.  A direct computation shows that $\omega$ belongs to the space
\[
A^{1,1}
=
\{\alpha\in\Omega^2(M)\mid I\alpha=\alpha,\;J\alpha=-\alpha,\;K\alpha=-\alpha\}
=
\Omega^{1,1}_I
\cap
\Omega^{2,0+0,2}_J
\cap
\Omega^{2,0+0,2}_K.
\]
Following \cite{salamon1986,banos2004}, elements of $A^{1,1}$ will be referred to as \emph{Salamon $2$-forms}.

Given $\alpha\in A^{1,1}$, the tensor
\[
g(X,Y):=\alpha(IX,Y)
\]
is symmetric, and it defines a hyperHermitian metric precisely when $\alpha$ is positive. Moreover, setting
\[
\alpha_J(X,Y):=-\alpha(KX,Y),
\qquad
\alpha_K(X,Y):=\alpha(JX,Y),
\]
one has
\[
\alpha_J,\alpha_K\in\Omega^{(2,0)+(0,2)}_I(M).
\]
\medskip \par 
Associated with a hyperHermitian metric we define an element of $\Omega^{2,0}_I$ via
\[
\Omega:=\frac{\omega_J+\sqrt{-1}\,\omega_K}{2},
\]
which precisely corresponds to 
\[
\Omega(X,Y):=\frac{-\omega(KX,Y)+\sqrt{-1} \omega(JX,Y)}{2}.
\]
The form $\Omega$ is nondegenerate and satisfies
\begin{equation}\label{eqn:Omega}
J\bar\Omega=\Omega,
\qquad
\Omega(JX,X)>0 \quad \forall\,X\neq 0,
\qquad
\frac{\Omega^n\wedge\bar\Omega^n}{n!}
=
\frac{\omega^{2n}}{(2n)!}.
\end{equation}
Conversely, a $(2,0)$-form $\Omega\in\Omega_I^{2,0}$ satisfying the first two conditions in \eqref{eqn:Omega} uniquely determines a hyperHermitian metric.  This observation motivates the following notion of positivity.

\begin{defn}\label{dfn:positive}
A form $\alpha\in\Omega_I^{2,0}(M)$ is called \emph{quaternionic positive} if
\[
J\bar\alpha=\alpha,
\qquad
\alpha(JX,X)>0
\]
for every nonzero vector $X\in TM$.
We write $\alpha>0$ whenever $\alpha$ is quaternionic positive.
\end{defn}
\begin{rmk}
Given a real valued function $\phi$, the operator $\del\del_J \phi$ satisfies
\[
J \overline{\del\del_J \phi}=J \bar\partial J \partial \phi= -\partial_J \partial \phi= \partial \partial_J \phi.
\]
In the light of definition \ref{dfn:positive}, we say that $\del\del_J \phi$ is quaternionic positive when 
\[
\del\del_J \phi(JX,X)>0.
\]
\end{rmk}
\noindent Finally, for $\alpha\in\Omega_I^{2,0}(M)$ we define its trace with respect to $\Omega$ by
\begin{equation}\label{eqn:tr_Omega}
\tr_\Omega\alpha
:=
n\,\frac{\alpha\wedge\Omega^{n-1}}{\Omega^n}.
\end{equation}
\begin{rmk} \label{rmk:standardcoordinates}
Let $(M,I,J,K,g)$ be a hyperHermitian manifold and let $p$ be a preferred point. We fix $I$-holomorphic coordinates $\{x_{2i}, x_{2i+1},y_{2i},y_{2i+1}\}_{i=0}^{n-1}$ such that $I \partial_{x_{i}}=\partial_{y_{i}}$. Let $\partial_{z_{i}}=\frac{\partial_{x_{i}}-\sqrt-1 \partial_{y_{i}}} {\sqrt2}$ and $dz^i= \frac{d{x^{i}}+\i d{y^{i}}} {\sqrt2}$. At the point of interest $p$ we may assume that $g$ is the identity and $J\partial_{z_{2i}}=\partial_{z_{\bar{2i+1}}}$. Such coordinates will be called \emph{standard}. Accordingly 
\[
\omega= \i \sum_{i=0}^{2n-1}dz^{\bar i i}, \ \ \Omega=\sum_{i=0}^{n-1}dz^{2i+1}\wedge dz^{2i}.
\]
Let $f$ be a smooth function. We define the Chern Laplacian 
\begin{equation} \label{eqn:chern_laplacian}
    \Delta_{\omega} f:= -2n \frac{\i \partial \bar \partial f \wedge \omega^{2n-1}}{\omega^{2n}}.
\end{equation}
It follows that at $p$, 
\[
\Delta_{\omega}f=\sum_{i=0}^{2n-1} f_{i\bar i}=\sum_{i=0}^{2n-1} \partial \bar \partial f(\partial_{z_{i}},\partial_{z_{\bar{i}}}).
\]
The choice of the minus sign in \eqref{eqn:chern_laplacian} is to guarantee that at the maximum (resp. minimum) of $f$, $\Delta_{\omega}f \le 0$ ($\ge 0$). 
In terms of the real coordinates, at the preferred point $p$ we have
\[
\Delta_{\omega}f=\tfrac12 \sum_{i=0}^{2n-1} dd^cf (\partial x_i, \partial y_i)=\tfrac12 \sum_{i=0}^{2n-1} dd^cf (\partial x_i, I\partial x_i). 
\]
Furthermore, using that $\partial\partial_J+\partial_J\partial=0$, 
\[
0=\partial J \bar \partial z^i-J \bar \partial J \partial z^i=-J \bar \partial J \partial z^i,
\]
we immediately obtain 
\begin{equation} \label{eqn:propertiesofJ}
    \bar \partial J  dz^i=0, \ \ \,  \partial J dz^{\bar{i}}=0. 
\end{equation}
This allow us to easily compute $\partial \partial_J f$, and a straightforward computation at the point $p$ shows that 
\[
\tr_\Omega \partial \partial_J f= -\Delta_{\omega}f.
\]
Therefore, we have the following equalities: for any function $f$
\[
\Delta_{\omega}f=\tfrac{1}{2} \sum_{i=0}^{2n-1}dd^c f (e_i,Ie_i)=-\tr_{\Omega}(\partial \partial_J f),
\]
where $\{e_i, Ie_i\}$ is any local orthonormal frame. 
\end{rmk}

\subsection{HKT metrics}

By now there are several equivalent definitions of a hyperK\"ahler with torsion (HKT) metric, which we state collectively:

\begin{defn} \label{defn:HKT}
Let $(M,I,J,K,g)$ be an hyperHermitian structure. We say that $(I,J,K,g)$ is \emph{hyperK\"ahler with torsion (HKT)} if one of the following equivalent conditions holds: 
\begin{enumerate} [label={(\arabic*)}]
    \item The form $\omega$ satisfies
\begin{equation} \label{eqn:Sal_diff}
d \omega \in \bigcap\limits_{L \in S^2} \left( \Omega^{2,1}_L(M) \oplus \Omega^{1,2}_L (M)\right).
\end{equation}
\item The Bismut connections associated with $(I,g)$, $(J,g)$, and $(K,g)$ coincide.
\item The HKT form $\Omega$ satisfies
\begin{align*}
    \partial \Omega = 0.
\end{align*}
\end{enumerate}
\end{defn}
\noindent Condition \eqref{eqn:Sal_diff} is equivalent to the vanishing of the Salamon differential of $\omega$ \cite{salamon1986,banos2004}. Although this formulation is not particularly convenient in practice, it is enough to verify that
\[
d\omega \in \bigcap_{L \in \{I,J,K\}} \Omega^{(2,1)+(1,2)}_L(M).
\]
While this criterion is often used implicitly in the literature, we include a proof in Lemma~\ref{l:3formlinearalgebra} as we did not find an explicit reference.  The second characterization is the one most commonly adopted in the literature (see, for instance, \cite{HowePap,grantcharov2000}), whereas the equivalence with the third was established in \cite{grantcharov2000}.

\medskip

Let $(I,J,K,g)$ be an HKT manifold, and let $\omega$, $\omega_J$, and $\omega_K$ denote the fundamental $2$-forms associated with the Hermitian structures $(I,g)$, $(J,g)$, and $(K,g)$. By Definition~\ref{defn:HKT}, the corresponding Bismut connections coincide. Equivalently, their torsion $3$-forms agree:
\[
d_I^c\omega
=
d_J^c\omega_J
=
d_K^c\omega_K.
\]
We shall denote this common torsion form by
\[
H:=d_I^c\omega=d_J^c\omega_J=d_K^c\omega_K.
\]

It is well known that $H$ is of type $(2,1)+(1,2)$ with respect to each of the complex structures $I$, $J$, and $K$. Furthermore, $dH$ is of type $(2,2)$ with respect to all three complex structures.

\medskip

We observe that HKT structures enjoy several remarkable properties:
\begin{enumerate} [label={(\arabic*)}]
    \item The holonomy group of the Bismut connection is contained in $\Sp(n)$;

    \item The Bismut Ricci form vanishes. This follows immediately from the previous item;

    \item The Lee forms of the Hermitian structures $(g,I)$, $(g,J)$, and $(g,K)$ coincide \cite{Ivanovstring}. We denote the common Lee form by $\theta$;

    \item For every $L\in\{I,J,K\}$, the Chern Ricci form of the Hermitian structure $(g,L)$ satisfies
    \begin{equation} \label{eqn:rhoChern}
        \rho_L=d(L\theta).
    \end{equation}
    With $\rho$ without any subscript we always mean $\rho_I$.
\item The Chern scalar curvatures satisfy
\(s_I=s_J=s_K\), so that we may unambiguously refer to the Chern scalar curvature \(s\). Indeed, this follows since $s_L =\tfrac 12 \sum_{i=1}^{4n} \rho_L(Le_i,e_i)$ and by Lemma \ref{l:trdItheta}
\begin{equation} \label{eqn:schern}
   s_L = \tfrac{1}{2}\sum_{i=1}^{4n} dL\theta(Le_i,e_i)
       =  d^*\theta+  |\theta|^2, 
\end{equation}
which is independent of the choice of complex structure
\(L\in\{I,J,K\}\). Furthermore when $M$ is compact by the formula above
\[
\int_M s \, dV \ge 0,
\]
and it is zero if and only if $\theta=0$.
    \item The curvature tensor $R^B \in \Omega^2(M) \otimes (\Omega_I^{1,1}(M) \cap \Omega_J^{1,1}(M) \cap \Omega_K^{1,1}(M))$, that is, 
    \[ R^B(X,Y,Z,W)=-R^B(Y,X,Z,W)=R^B(X,Y,LZ,LW),\]
    for any $L \in \{I,J,K\}$.
    \item From the previous item, if we consider the second Bismut Ricci curvature $S^B$ of $(g,I)$, then $S^B$ is $(1,1)$ with respect to $I,J,K$. \label{item:6}
    \item The $(2,0)$ form $\Omega$ defines a class in $H^{2,0}_{qBC}(M)$. By definition \ref{defn:HKT}, 
    \[
    \partial \Omega=0, \ \ \partial_J \Omega=J \bar \partial J \Omega=J \bar \partial \bar \Omega=0,
    \]
    where we used that $J\bar \Omega=\Omega$.
\end{enumerate}

\subsection{HKT-Einstein metrics} \label{ss:HKTE}

The search for canonical metrics in HKT geometry naturally leads to an analogue of the Kähler–Einstein condition. Following a suggestion of Verbitsky (cf. \cite{fusi2026}), an HKT metric is called HKT–Einstein when the the $J$-anti-invariant component of the first Chern–Ricci form is proportional to the fundamental form. \par
Given $g$ a hyperHermitian metric consider the $2$-form
\begin{equation} \label{f:phi}
    \Phi(g) := \frac{\rho - J\rho}{2} \in A^{1,1}.
\end{equation}

\begin{defn} \label{defn:HKT_Einst} We say that an HKT structure $(I, J, K,g)$ is \emph{HKT-Einstein} if there exists $\lambda \in \mathbb R$ such that
\begin{align*}
    \Phi(g) = \lambda \omega.
\end{align*}
The constant $\lambda$ will be called HKT-Einstein constant. 
\end{defn}
Tracing both sides of this equation with $\omega$ yields
\begin{equation} \label{eqn:chern_scalar}
    s=2n \lambda.
\end{equation} 
However, by \cite[Lemma 2.10]{fusi2026}, when $M$ is compact, $s$ must be non-negative, implying that the same property holds for the HKT-Einstein constant $\lambda$. 

\subsection{First quaternionic Bott-Chern class}

Let $(M,I,J,K,g)$ be a hyperHermitian manifold, and consider the $n$-form $\bar{\Omega}^n \in \Omega^{(0,2n)}_I$. We define the $(1,0)$-form $\alpha$ by
\begin{align} \label{f:alphadef}
\partial \bar\Omega^n = \alpha \wedge \bar{\Omega}^n.
\end{align}
It is well known that $\alpha + \bar\alpha$ is the connection $1$-form of the Obata connection on the canonical bundle of $I$.  By \cite[Section 10.01]{Ver02} $\alpha$ is $\partial$-closed, and hence $\partial \partial_J$-closed and $J\overline{\del_J \alpha}=\partial_J\alpha$, that is, $\del_J \alpha$ satisfies the first condition of definition \ref{dfn:positive}. 

Furthermore, when the hyperHermitian structure is HKT, $\alpha$ is the $(1,0)$-part of the Lee form, that is,  
\begin{equation} \label{eqn:leeform}
    \theta=\alpha + \bar \alpha. 
\end{equation}
\noindent We now describe $\del_J \alpha$ in terms of classical curvatures.  
It follows from \cite[Proposition~2.6(c)]{fusi2026}\footnote{Note that here there is a minus sign, which arises from the different convention used for the action of the complex structure on differential forms.} that 
\begin{equation} \label{eqn:del_Jalpha}
\Phi(g) := \frac{\i \Phi(g)(JX, Y) -\Phi(g)(KX, Y)}{2}  =- \partial_J \alpha.
\end{equation}
Furthermore observe using
\eqref{eqn:del_Jalpha} that
\begin{equation}\label{eqn:cherscalar}
\operatorname{tr}_{\Omega}(\partial_J\alpha)= -\tfrac 12 s.    
\end{equation}

From the discussion above and equation \eqref{eqn:del_Jalpha}, it is natural to consider the object $\del_J \alpha$ as a natural analogue of Ricci curvature in the HKT setting.  It turns out that this object gives a well-defined element of quaternionic Bott-Chern cohomology.  To see this we first observe an elementary transgression formula for $\alpha$ in the lemma below, where the notation $\alpha_{\Xi}$ denotes the quantity defined by (\ref{f:alphadef}) for a form $\Xi \in \Omega^{(2n,0)}$.
\begin{lem} \label{l:transgression} 
Given $(M,I,J,K,g)$ a hyperHermitian manifold, $\Xi \in \Omega^{2n,0}_I$ nowhere vanishing and $f \in C^{\infty}(M)$, 
\begin{align*}
    \alpha_{e^{f} \Xi} = \alpha_{\Xi} + \, \del f.
\end{align*}
\end{lem}
\begin{proof}
We compute
\begin{align*}
\partial (e^{f} \bar{\Xi}) =&\ \partial(e^{f})\wedge\bar{\Xi}
  +e^{f}\partial \bar{\Xi} = \left( \del f + \alpha_{\Xi} \right) \wedge (e^{f}\bar{\Xi}).
\end{align*}
\end{proof}

\begin{defn}
The first quaternionic Bott--Chern class of the hypercomplex structure is defined as
\[
c^{\mathrm{qBC}}_1(M, S) := [-\partial_J \alpha]_{\mathrm{qBC}},
\]
well-defined for any choice of hyperHermitian metric $g$ by Lemma \ref{l:transgression}.
\end{defn}
\begin{rmk}
A sign discrepancy exists between our definition and the one given, for instance, in \cite[Definition 3.1]{fusi2026}, which is due to different conventions. We point out, however, that the definition of positivity remains consistent. We say that the first quaternionic Bott--Chern class of the hypercomplex structure is positive if there exists a hyperHermitian metric such that $-\partial_J \alpha$ is positive in the sense of Definition~\ref{dfn:positive}, that is, 
\[
-\partial_J \alpha (JX,X) >0.
\]
Observe that this is equivalent to 
\[
\partial_J \alpha (X,JX) >0,
\]
which is precisely the definition of positivity adopted in \cite{fusi2026}.
\end{rmk}
\section{HKT Ricci flow} \label{s:HKTflow}

In this section we develop the fundamental theory of HKT Ricci flow, leading to the proof of Theorem \ref{t:mainthm1}.

\subsection{Curvature identities}

The basic motivation for the definition of our flow comes from the following fundamental curvature identity of Ivanov-Papadopoulos which relates Bismut and Chern Ricci curvatures.  

\begin{lem}\cite[Proposition 3.3, Proposition 3.16]{Ivanovstring}.\label{l:Ricciswap} Let $(M^{2n}, I, g)$ be a Hermitian manifold.  Then
\begin{align} \label{f:Ricciswap}
S = (\rho^B)^{1,1} + Q^1 - \tfrac{1}{4} \lambda^{\omega},
\end{align}
where
\begin{align*}
    Q^1(X,Y) = g ( i_X T, i_{IY} T) ,
\end{align*}
and $T$ is the torsion of the Chern connection of $(I,g)$.
Furthermore 
\begin{equation} \label{eqn:grf}
 \tfrac14 \lambda^{\omega}(X,Y)- \rho^B(X,Y) = -\bigl( \mathrm{Rc}^B(X,IY) + (\nabla^B_X \theta) IY \bigr).   
\end{equation}
\end{lem}

From this lemma a naive idea emerges.  For an HKT structure we know that $\rho^B$ vanishes. As we will show in the next lemma, $\lambda^{\omega}$ defines a Salamon $2$-form, hence by Lemma \ref{l:Ricciswap} appears to define a natural elliptic operator on the space of HKT structures, as $S$ is elliptic.  However to properly define a parabolic flow of HKT structures we need this form to also be closed under the Salamon differential.  To achieve this we must modify by a delicate further choice of lower order quadratic torsion term.  The rest of the section develops curvature identities building towards this point which we prove in Proposition \ref{p:Phiprops}.

\begin{lem} \label{l:trdHalgebraic} Given $( I, J, K,g)$ HKT, we have
\begin{align*}
    \lambda^{\omega} \in \Omega^{1,1}_I \cap \Omega^{2,0 + 0,2}_J \cap \Omega^{2,0 + 0,2}_K.
\end{align*}
\end{lem}
\begin{proof}  
    As $dH \in \Omega^{2,2}_I$, it follows easily that  $\mathrm{tr}_{\omega} dH \in \Omega^{1,1}_I$.  We further compute 
\begin{align*}
\lambda^{\omega}(JX,JY) =&\ 
\sum_{i=1}^{4n} dH(JX,JY,e_i,Ie_i) \\
=&\ 
\sum_{i=1}^{4n} dH(X,Y,Je_i,JIe_i) \\
=&\ 
-\sum_{i=1}^{4n} dH(X,Y,Je_i,IJe_i) \\
=&\ 
-\lambda^{\omega}(X,Y),
\end{align*}
where in the second equality we used the fact that $dH \in \Omega^{2,2}_J$, and in the third equality the quaternionic relation $JI=-IJ$. The argument for $K$ is analogous.
\end{proof}

\begin{prop} 
Let $( I, J, K,g)$ be a HKT structure.  Then
\begin{equation} \label{eqn:trdH}
\tfrac{1}{4}\lambda^{\omega} = S^B - \rho + H_I^2,
\end{equation}
where
\begin{equation} \label{f:HIdef}
H_I^2(X,Y)
:=
\sum_{i=1}^{4n} g\bigl(H(X,Ie_i),H(Y,e_i)\bigr),
\end{equation}
where $\{e_i\}$ is a local $g$-orthonormal frame.
\end{prop}
\begin{proof}
We use known identities for the Hermitian structure $(g, I)$.  Let $H=d^c_I \omega$. Let $\nabla^+=\nabla^B$ and let $\nabla^-$ denote the metric connection with opposite torsion, that is 
\[
g(\nabla^{\pm}_X Y, Z)=g(\nabla_X Y, Z) \pm \tfrac{1}{2} d^c_I \omega (X,Y,Z),
\]
where $\nabla$ is the Levi-Civita connection of $g$. 
The curvature tensors $R^\pm$ have the following expression (see, for instance \cite[Proposition 3.18]{GRFbook}):
\begin{align*}
R^{\pm}(X,Y,Z,W)=&\ R(X,Y,Z,W)
\pm \tfrac{1}{2} \nabla_X H(Y,Z,W) \mp \tfrac{1}{2} \nabla_Y H(X,Z,W) \\
&-\tfrac{1}{4} \langle H(X,W),H(Y,Z)\rangle+ \tfrac{1}{4} \langle H(Y,W),H(X,Z)\rangle.
\end{align*}
This yields the consequence
\[
R^-(X,Y,Z,W)=R^+(X,Y,Z,W)-  \nabla_X H(Y,Z,W) +  \nabla_Y H(X,Z,W).
\]
By \cite[Theorem 1.6]{Bismut}, 
\begin{align*}
\tfrac{1}{2} dH(X,Y,Z,W)=&\ R^+(X,Y,Z,W)-R^-(Z,W,X,Y)\\
=&\ R^+(X,Y,Z,W)-R^+(Z,W,X,Y)+  \nabla_Z H(W,X,Y) - \nabla_W H(Z,X,Y).
\end{align*}
We exploit the above equality to obtain the desired expression for $\lambda^{\omega}$:
\begin{align*}
\tfrac{1}{2} \lambda^{\omega} (Z,W) =&\ \tfrac{1}{2}\sum_{i=1}^{4n} dH(e_i, Ie_i, Z,W)\\
=&\ \sum_{i=1}^{4n} R^B(e_i,Ie_i,Z,W)-R^B(Z,W,e_i,Ie_i)+  \nabla_Z H(W,e_i,Ie_i) -  \nabla_W H(Z,e_i,Ie_i)\\
=&\ 2 S^B(Z,W)-2\rho^B(Z,W) + \sum_{i=1}^{4n}\nabla_Z H(W,e_i,Ie_i) -  \nabla_W H(Z,e_i,Ie_i)\\
=&\ 2 S^B(Z,W)+\sum_{i=1}^{4n}\nabla_Z H(W,e_i,Ie_i) -  \nabla_W H(Z,e_i,Ie_i),
\end{align*}
where $\rho^B$ vanishes since the structure is HKT. \par
On the other hand, 
\[
\sum_{i=1}^{4n}\nabla_Z H(W,e_i,Ie_i)=-2 \nabla^B_Z I\theta (W)-\iota_{I\theta^\sharp}H(Z,W)+g(H(Z,Ie_i), H(W,e_i)),
\]
which implies that 
\[
\tfrac{1}{2} \lambda^{\omega} (Z,W)=2 S^B(Z,W)-2 d^{\nabla^{B}} I\theta (Z,W)-2 \iota_{I\theta^\sharp}H(Z,W)+ 2 g(H(Z,Ie_i), H(W,e_i)),
\]
where $d^{\nabla^{B}}$ is the differential respect to $\nabla^B$. Using that $\nabla^B=\nabla+\tfrac{1}{2}H$, 
\begin{equation} \label{eqn:dItheta}
d^{\nabla^{B}} I\theta(Z,W)=dI\theta(Z,W)-\iota_{I\theta^\sharp}H(Z,W)=\rho (Z,W)-\iota_{I\theta^\sharp}H(Z,W), 
\end{equation}
implying that 
\[
\tfrac{1}{2} \lambda^{\omega} (Z,W)=2 S^B(Z,W)-2 \rho (Z,W)+ 2 g(H(Z,Ie_i), H(W,e_i)),
\]
as claimed.
\end{proof}

\begin{lem} \label{l:Qalgebraic} Given $(I, J, K,g)$ an HKT structure, define
\begin{align*}
    Q^2 = H_I^2 - J H_I^2,
\end{align*}
where $H_I^2$ is defined as in \eqref{f:HIdef}.  Then $Q^2 \in A^{1,1}$.
\begin{proof}Equation \eqref{eqn:trdH} directly implies $I H_I = H_I$: all other terms lie in $\Omega^{1,1}_I$, so $H_I$ must as well. Moreover, since $I$ and $J$ anticommute, the tensor $Q^2$ also belongs to $\Omega^{1,1}_I$, while by construction $Q^2 \in \Omega^{2,0+0,2}_J$. Using this and the quaternionic relations, it follows easily that $Q^2 \in \Omega^{2,0 + 0,2}_K$, as claimed.    
\end{proof}
\end{lem}

\subsection{Ricci curvatures and \texorpdfstring{$\Phi$}{Phi}}

With the above preliminaries in place we are ready to unpack the geometric structure of the operator $\Phi$ introduced in \eqref{f:phi}.

\begin{prop} \label{p:PhiasChern} Given $(I, J, K,g)$ an HKT structure, the operator $\Phi(g)$ can be expressed equivalently as:
\begin{enumerate} [label={(\arabic*)}]
\item $ \Phi(g) = - \tfrac{1}{4} \lambda^{\omega} + \tfrac12 Q^2$.
\item $ \Phi(g) = S - Q, $ where
\begin{equation}\label{eqn:Q3}
    Q(X,Y) := \left( Q^1-\frac{1}{2}Q^2 \right)(X,Y) = g ( i_X T, i_{IY} T)
-\tfrac{1}{2} \big(H_I^2(X,Y)-JH_I^2(X,Y)\big).
\end{equation}
\item $\Phi(g) = -S^B +\rho_I^C-\frac{1}{2} (H_I^2 + J H_I^2)$.
\item $\Phi( g)(IX,Y)=\mathrm{Rc}-\tfrac14 H^2 +\tfrac12 \mathcal{L}_{\theta^\sharp}g +\tfrac 12 \hat Q^2$, where 
\begin{equation}  \label{eqn:Q_hat}
\hat Q^2(X,Y):=H^2_I (IX,Y)-H^2_I(JX,KY).
\end{equation}
\end{enumerate}
\end{prop}
\begin{proof}
Since $\lambda^{\omega}$ is of type $(2,0)+(0,2)$ with respect to $J$, and $S^B$ is $(1,1)$ with respect to each complex structure $I,J,K$,  Equation~\eqref{eqn:trdH} yields
\[
\frac{1}{4}\lambda^{\omega}
=-\,\Phi(g)+\frac{1}{2}\,Q^2.
\]
Hence,
\[
\Phi(g)=-\frac14\,\lambda^{\omega}+\frac12\,Q^2,
\]
which proves claim~(1).

Applying Lemma~\ref{l:Ricciswap} to the above identity, we obtain
\[
\Phi(g)=S-Q^1+\frac12\,Q^2,
\]
proving claim~(2).

Finally, combining claim~(1) with Equation~\eqref{eqn:trdH}, we find
\[
\Phi(g)
=-\frac14\,\lambda^{\omega}+\frac12\,Q^2
=-S^B+\rho_I^C-H_I^2+\frac12\,Q^2
=-S^B+\rho_I^C-\frac12\bigl(H_I^2+JH_I^2\bigr),
\]
which proves claim~(3).

We now prove the last claim. By item (1),
\[
\Phi(g)(IX,Y) = -\tfrac14 \lambda^{\omega}(IX,Y) + \tfrac12 Q^2(IX,Y).
\]
By Equation \eqref{eqn:grf} and the fact that the Bismut Ricci form vanishes in the HKT setting,
\[
\tfrac14 \lambda^{\omega}(X,Y) = -\bigl( \mathrm{Rc}^B(X,IY) + (\nabla^B_X \theta) IY \bigr).
\]
Furthermore, since $\lambda^{\omega}$ is of type $(1,1)$ with respect to $I$, we have
\[
\tfrac14 \lambda^{\omega}(IX,Y) = -\tfrac12 \Bigl( \mathrm{Rc}^B(IX,IY) + (\nabla^B_{IX} \theta) IY + \mathrm{Rc}^B(X,Y) + (\nabla^B_X \theta) Y \Bigr).
\]
From \cite[Equation 3.26]{Ivanovstring},
\[
\mathrm{Rc}^B(IX,IY) = \mathrm{Rc}^B(Y,X) - (\nabla^B_{IX} \theta) IY + (\nabla^B_Y \theta) X,
\]
and therefore
\[
\tfrac14 \lambda^{\omega}(IX,Y) = - \Bigl( \mathrm{Rc}^{B}_{\mathrm{sym}} + \tfrac12 \mathcal{L}_{\theta^\sharp} g \Bigr)(X,Y).
\]
On the other hand,
\[
Q^2(IX,Y) = H^2_I(IX,Y) - J H^2_I(IX,Y) = H^2_I(IX,Y) + H^2_I(KX,JY) = H^2_I(IX,Y) - H^2_I(JX,KY),
\]
where the last equality holds because $H^2_I$ is of type $(1,1)$ with respect to $I$.

Consequently,
\[
\Phi(g) (I \cdot, \cdot)=  \Bigl( \mathrm{Rc} - \tfrac14 H^2 + \tfrac12 \mathcal{L}_{\theta^\sharp} g \Bigr) + \tfrac12 \hat Q^2.
\]
\end{proof}

\begin{lem} \label{l:Q_1}
Given an HKT metric,
\begin{enumerate} [label={(\arabic*)}]
    \item One has 
    \[
Q (X,Y)=\tfrac{1}{4} \big[H^2(X,IY)-H^2(IX,Y)\big]- H^2_I (X,Y)+ \tfrac{1}{2} JH^2_I (X,Y).
\]
\item $\tr_{\omega} Q = 0$.
\end{enumerate}

\end{lem}
\begin{proof} 
We first establish that 
\[
Q^1(X,Y)=g(i_XT,i_{IY}T)
=\tfrac14\bigl(H^2(X,IY)-H^2(IX,Y)\bigr)-\tfrac12 H_I^2(X,Y).
\]
To see this, note that by Equation~\eqref{eqn:chern_torsion},
\begin{align*}
4\, g(i_XT,i_{IY}T)
&=\bigl(H(X,Ie_i,e_j)+H(IX,e_i,e_j)\bigr)
  \bigl(H(IY,Ie_i,e_j)-H(Y,e_i,e_j)\bigr) \\
&= H^2(X,IY)-H_I^2(X,Y)-H_I^2(IX,IY)-H^2(IX,Y).
\end{align*}
Item (1) follows from the definition of $Q$ given in equation \eqref{eqn:Q3} and the \(I\)-invariance of \(H_I^2\). \par
To prove the second claim we compute 
\begin{align*}
    -\sum_{i=1}^{4n} Q(e_i,Ie_i)&=\sum_{i=1}^{4n}\tfrac12 H^2(e_i,e_i)+\tfrac32 H^2_I(e_i,Ie_i) =\tfrac12 |H|^2+\tfrac32\sum_{i=1}^{4n} H^2_I(e_i,Ie_i).
\end{align*}
To conclude the proof, we need to prove that 
\[
\sum_{i=1}^{4n}H^2_I(e_i,Ie_i)=-\tfrac13 |H|^2.
\]
First, exploiting that $H$ is $(2,1)+(1,2)$ with respect to $I$ (see Equation \eqref{eq:type_condition}), 
\[
\begin{split}
\sum_{i,j,k=1}^{4n} H(e_i,e_j,e_k)H(Ie_i,Ie_j,e_k)&=\sum_{i,j,k=1}^{4n} H(e_i,e_j,e_k)H(e_i,e_j,e_k)-H(e_i,e_j,e_k)H(e_i,Ie_j,Ie_k)\\
&\ \ \   - H(e_i,e_j,e_k)H(Ie_i,e_j,Ie_k)\\
&=|H|^2 -2 \sum_{i,j,k=1}^{4n} H(e_i,e_j,e_k)H(Ie_i,Ie_j,e_k),
\end{split}
\]
that is 
\[
3 \sum_{i,j,k=1}^{4n}  H(e_i,e_j,e_k)H(Ie_i,Ie_j,e_k)= |H|^2.
\]
Furthermore
\begin{align*}
\sum_{i=1}^{4n}H^2_I(e_i,Ie_i)&=\sum_{i,j,k=1}^{4n} H(e_i,Ie_j,e_k) H(Ie_i,e_j,e_k)\\
&=-\sum_{i,j,k=1}^{4n} H(e_i,e_j,e_k) H(Ie_i,Ie_j,e_k)\\
&=-\tfrac13|H|^2,
\end{align*}
giving the desired identity. 
\end{proof}

\subsection{Integrability of \texorpdfstring{$\Phi$}{Phi}}

Next, using these curvature identities we establish several remarkable properties of the operator $\Phi$, in particular showing the key point that it is a tangent vector to the space of HKT structures.  We first need a preliminary lemma which is presumably known to experts:

\begin{lem} \label{l:3formlinearalgebra}
Let $(M,I,J,K)$ be a hypercomplex manifold, and let $\alpha\in \Omega^3(M)$. 
Assume that $\alpha$ is of type $(1,2)+(2,1)$ with respect to each of the complex structures $I$, $J$, and $K$. Then $\alpha$ is of type $(1,2)+(2,1)$ with respect to every induced complex structure
\[
L=aI+bJ+cK,
\qquad a^2+b^2+c^2=1.
\]
\end{lem}

\begin{proof}
Recall that a real $3$-form $\alpha$ is of type $(1,2)+(2,1)$ with respect to a complex structure $L$ if and only if
\begin{equation}\label{eq:type_condition}
\alpha(X,Y,Z)
=
\alpha(LX,LY,Z)
+\alpha(LX,Y,LZ)
+\alpha(X,LY,LZ)
\end{equation}
for all vector fields $X,Y,Z$.

Fix an induced complex structure
\[
L=aI+bJ+cK,
\qquad a^2+b^2+c^2=1.
\]
Expanding the right-hand side of \eqref{eq:type_condition}, we obtain
\begin{align*}
&\alpha(LX,LY,Z)
+\alpha(LX,Y,LZ)
+\alpha(X,LY,LZ)=
a^2A_I+b^2A_J+c^2A_K
+ab\,B_{IJ}
+bc\,B_{JK}
+ca\,B_{KI},
\end{align*}
where
\begin{align*}
A_I
&=
\alpha(IX,IY,Z)
+\alpha(IX,Y,IZ)
+\alpha(X,IY,IZ), \\
A_J
&=
\alpha(JX,JY,Z)
+\alpha(JX,Y,JZ)
+\alpha(X,JY,JZ), \\
A_K
&=
\alpha(KX,KY,Z)
+\alpha(KX,Y,KZ)
+\alpha(X,KY,KZ),
\end{align*}
and
\begin{align*}
B_{IJ} =&\
\alpha(IX,JY,Z)
+\alpha(JX,IY,Z)
+\alpha(IX,Y,JZ)\\
&\ 
+\alpha(JX,Y,IZ)
+\alpha(X,IY,JZ)
+\alpha(X,JY,IZ), \\ 
B_{JK}
=&\ 
\alpha(JX,KY,Z)
+\alpha(KX,JY,Z) 
+\alpha(JX,Y,KZ)\\
&\ 
+\alpha(KX,Y,JZ) 
+\alpha(X,JY,KZ)
+\alpha(X,KY,JZ), \\
B_{KI}
=&\ 
\alpha(KX,IY,Z)
+\alpha(IX,KY,Z) 
+\alpha(KX,Y,IZ)\\
&\ 
+\alpha(IX,Y,KZ) 
+\alpha(X,KY,IZ)
+\alpha(X,IY,KZ).
\end{align*}
Since $\alpha$ is of type $(1,2)+(2,1)$ with respect to $I$, $J$, and $K$, equation \eqref{eq:type_condition} gives
\[
A_I=A_J=A_K=\alpha(X,Y,Z).
\]
Hence 
\begin{align}
&\alpha(LX,LY,Z)
+\alpha(LX,Y,LZ)
+\alpha(X,LY,LZ)=\alpha(X,Y,Z)
+ab\,B_{IJ}
+bc\,B_{JK}
+ca\,B_{KI},
\label{eq:reduced}
\end{align}
where we have used the identity $a^2+b^2+c^2=1$. It remains to prove that
\[
B_{IJ}=B_{JK}=B_{KI}=0.
\]
By cyclic symmetry, it is enough to show that $B_{IJ}=0$.  Using \eqref{eq:type_condition} for the complex structure $I$, evaluated at $(X,JY,Z)$, we obtain
\[
\alpha(IX,JY,Z)
=
-\alpha(X,KY,Z)
-\alpha(X,JY,IZ)
+\alpha(IX,KY,IZ).
\]
Substituting this identity into the expression for $B_{IJ}$ yields
\begin{gather}
\begin{split}
B_{IJ}
=&\ 
-\alpha(X,KY,Z)
+\alpha(IX,KY,IZ)
+\alpha(JX,IY,Z)\\
&\ 
+\alpha(IX,Y,JZ)
+\alpha(JX,Y,IZ)
+\alpha(X,IY,JZ).
\label{eq:bij_step1}
\end{split}
\end{gather}

Next, applying \eqref{eq:type_condition} with respect to $J$, evaluated at $(X,IY,Z)$, we find
\[
\alpha(X,IY,JZ)
=
-\alpha(JX,KY,JZ)
-\alpha(JX,IY,Z)
+\alpha(X,KY,Z).
\]
Substituting this into \eqref{eq:bij_step1}, we obtain
\begin{align}
B_{IJ}
&=
\alpha(IX,KY,IZ)
+\alpha(IX,Y,JZ) 
+\alpha(JX,Y,IZ)
-\alpha(JX,KY,JZ).
\label{eq:bij_step2}
\end{align}
Finally, applying \eqref{eq:type_condition} with respect to $K$, evaluated at $(IX,Y,JZ)$, gives
\[
\alpha(IX,Y,JZ)
=
\alpha(JX,KY,JZ)
-\alpha(JX,Y,IZ)
-\alpha(IX,KY,IZ).
\]
Substituting this identity into \eqref{eq:bij_step2}, we conclude that
\[
B_{IJ}=0.
\]
Therefore \eqref{eq:reduced} reduces to
\[
\alpha(LX,LY,Z)
+\alpha(LX,Y,LZ)
+\alpha(X,LY,LZ)
=
\alpha(X,Y,Z),
\]
which is precisely the condition that $\alpha$ be of type $(1,2)+(2,1)$ with respect to $L$.
\end{proof}

\begin{prop} \label{p:Phiprops}
Let $(g, I, J, K)$ be an HKT structure. Then
\begin{enumerate} [label={(\arabic*)}]
    \item $\Phi(g) \in A^{1,1}$,
    \item $d \Phi(g) \in \bigcap\limits_{L \in S^2} \left( \Omega^{2,1}_L (M)\oplus \Omega^{1,2}_L (M) \right)$.
\end{enumerate}
\end{prop}

\begin{proof}
The first claim is immediate from Lemmas \ref{l:trdHalgebraic} and \ref{l:Qalgebraic}.  We verify the second claim.  First note that by using Lemma \ref{l:3formlinearalgebra}, it suffices to show that $d \Phi \in \Omega^{2,1 + 1,2}_L$ for $L = I, J, K$.  The claim is obvious for $I$ from item (1).

Turning to the complex structure $J$, we observe using item (3) in Proposition \eqref{p:PhiasChern} that
\begin{align*}
-\Phi(g) = \tfrac{1}{4}\lambda^{\omega}
-\tfrac12\bigl(H_I^2-JH_I^2\bigr)
=
S^B-\rho
+\tfrac12\bigl(H_I^2+JH_I^2\bigr).
\end{align*}
Since $\rho$ is closed, we obtain
\[
d\left( \tfrac{1}{4}
\lambda^{\omega}
-\tfrac12\bigl(H_I^2-JH_I^2\bigr)
\right)
=
d\left(
S^B+\tfrac12(H_I^2+JH_I^2)
\right).
\]
Using that the Bismut connection preserves every complex structure, it follows that $S^B \in \Omega^{1,1}_J$ (see item \ref{item:6}).  Furthermore, it is clear by construction that 
$\tfrac12(H_I^2+JH_I^2) \in \Omega^{1,1}_J$.  Hence $S^B+\tfrac12(H_I^2+JH_I^2) \in \Omega^{1,1}_J$, and thus its exterior derivite lies in $\Omega^{2,1 + 1,2}_J$, as claimed.

The argument is similar for the complex structure $K$.  The only missing point is to show that the torsion term $H_I^2+JH_I^2 \in \Omega^{1,1}_K$.  We compute using the quaternion relations and that $H_I^2 \in \Omega^{1,1}_I$,
\begin{align*}
(H_I^2& +JH_I^2)(X,Y)
-(H_I^2+JH_I^2)(KX,KY)\\
=&\ 
H_I^2(X,Y)
+H_I^2(JX,JY) -H_I^2(KX,KY)
-H_I^2(JKX,JKY)\\
=&\ H_I^2(X, Y) + H_I^2(JX, JY) - H_I^2 (IJ X, IJ Y) - H_I^2(IX, IY)\\
=&\ 0,
\end{align*}
as claimed.
\end{proof}

\subsection{Proof of Theorem \ref{t:mainthm1}}
\begin{proof}[Proof of Theorem \ref{t:mainthm1}]

As hyperHermitian curvature flow is a special case of Hermitian curvature flow as defined in \cite{HCF}, items (1) and (2) follow from \cite[Theorem 1.1]{HCF}.  To show item (3), we observe that the general short-time existence theory for quasilinear parabolic systems can be set up for equation (\ref{f:Omegaflow}) in an appropriate Banach space of HKT forms, noting further that the operator defined in Definition \eqref{f:phi} is a tangent vector to the space of HKT metrics.  As this operator is strictly elliptic for HKT metrics by Proposition \ref{p:PhiasChern}, it follows there exists a unique solution to to (\ref{f:Omegaflow}).  Alternatively, we show in \S \ref{ss:scalarreduction} below that the solution to (\ref{f:Omegaflow}) can be reduced to a scalar parabolic Monge-Amp\`ere type flow which admits unique short-time solutions.  Noting that solutions to (\ref{f:Omegaflow}) induce solutions to hyperHermitian curvature flow by Proposition \ref{p:PhiasChern}, and such solutions are unique, items (3a) and (3b) follow.  Item (3c) then follows directly from Proposition \ref{p:PhiasChern}.
\end{proof}

\subsection{Scalar curvature monotonicity} 

In this subsection we prove the monotonicity formula for Chern scalar curvature.  We first record a preliminary lemma computing an evolution equation for the Lee form.  The formula suffices for our purposes but is not written as a heat-type equation, such a formulation following from further curvature identities for HKT structures.
\begin{lem} \label{l:leeform}
Given $g_t$ a solution to the HKT Ricci flow, 
\begin{align}
\partial_t \theta&= -\tfrac 12 ds.
\end{align}
\end{lem}
\begin{proof}
Differentiating equation (\ref{f:alphadef}) and using (\ref{f:Omegaflow}) we obtain
\begin{align*}
\partial_t(\partial \bar\Omega^n)
    &= \partial_t(\alpha \wedge \bar\Omega^n) \\
    &= (\partial_t\alpha)\wedge \bar\Omega^n
       + \alpha \wedge \partial_t(\bar\Omega^n)\\
    &= (\partial_t\alpha)\wedge \bar\Omega^n
       + n\,\alpha \wedge \overline{\partial_J\alpha}
         \wedge \bar\Omega^{n-1}\\
         &=
    \left(
        \partial_t\alpha
        +
        \overline{\operatorname{tr}_{\Omega}(\partial_J\alpha)}\,\alpha
    \right)
    \wedge \bar\Omega^n,
\end{align*}
where in the final line we used \eqref{eqn:tr_Omega}.  On the other hand, again using (\ref{f:Omegaflow}) and \eqref{eqn:tr_Omega},
\begin{align*}
\partial_t(\partial \bar\Omega^n)
    &= \partial\bigl(\partial_t \bar\Omega^n\bigr) \\
    &= \partial\bigl(
        n\,\overline{\partial_J\alpha}
        \wedge \bar\Omega^{n-1}
      \bigr) \\
    &= \partial\bigl(
       \overline{ \operatorname{tr}_{\Omega}(\partial_J\alpha)}
        \,\bar\Omega^n
      \bigr) \\
    &= \partial\!\left(
         \overline{ \operatorname{tr}_{\Omega}(\partial_J\alpha)}
      \right)
      \wedge \bar\Omega^n
      +
       \overline{ \operatorname{tr}_{\Omega}(\partial_J\alpha)}\,
      \partial\bar\Omega^n \\
    &=  \left(
        \partial\!\left(
             \overline{ \operatorname{tr}_{\Omega}(\partial_J\alpha)}
        \right)
        +
         \overline{ \operatorname{tr}_{\Omega}(\partial_J\alpha)}\,\alpha
      \right)
      \wedge \bar\Omega^n.
\end{align*}
Comparing these two expressions above yields
\begin{equation*}
  \partial_t \alpha=  \partial\!\left(
             \overline{ \operatorname{tr}_{\Omega}(\partial_J\alpha)}
        \right) . 
\end{equation*}
The result follows from equation (\ref{eqn:cherscalar}), using that $\theta = \alpha + \bar{\alpha}$.
\end{proof}

We are now ready to prove the scalar curvature montonicity.  In the theorem below we use the notation
\begin{align*}
    \square = \dt - \Delta_{\omega_t}
\end{align*}
for the heat operator associated to the time-dependent Chern connection.

\begin{thm} \label{t:scalarthm_bulk}
 Given $g_t$ a solution to the HKT Ricci flow, 
\begin{align} \label{eq:evolution_s}
\square s=\tfrac12 |\Phi(g)|^2.
\end{align}
Assuming $M$ is compact, 
\begin{enumerate} [label={(\arabic*)}]
    \item $\inf_{M \times \{t\}} s \geq \inf_{M \times \{0\}} s$,
    \item $\inf_{M \times \{t\}} s \geq - \tfrac{2n}{t}$,
    \item If $\inf_{M \times \{0\}} s = \sigma > 0$, then the maximal smooth existence time of the flow is $T \leq \tfrac{2n}{\sigma}$.
\end{enumerate}
\end{thm}

\begin{proof} We first show equation \eqref{eq:evolution_s}. By definition of $\Phi(g)$ we may express
\[
s=\frac{1}{2}\sum_{i=1}^{4n} \rho (Ie_i,e_i)=\tr_{\omega}\rho=\tr_{\omega}(\Phi (g)).
\]
Therefore by Lemma \ref{l:leeform}:
\begin{align*}
\partial_t s&=- \tr_{\omega} \left( \partial_t \omega \,  \omega^{-1} \,  \Phi(g) \right) + \tr_{\omega} (\partial_t \Phi (g))\\
&= \tfrac12 |\Phi(g)|^2+\tr_{\omega}(dI \partial_t \theta)\\
&=  \tfrac12 |\Phi(g)|^2-\tfrac12 \tr_{\omega}(dI ds)\\
&=  \tfrac12 |\Phi(g)|^2+\Delta_{\omega} (s),
\end{align*}
where we used that $\rho=dI\theta$.

Now assuming $M$ is compact, the first item follows by a direct application of the maximum principle to \eqref{eq:evolution_s}.  For the second and third items we first observe the algebraic fact using \eqref{eqn:chern_scalar} that
\begin{align*}
    \tfrac 12 \brs{\Phi(g)}^2 =&\tfrac 12 \brs{\mathring{\Phi(g)} + \tfrac{s}{2n} \omega }^2 = \tfrac12 \brs{\mathring{\Phi(g)}}^2 + \tfrac{1}{2n} s^2 \geq \tfrac{1}{2n} s^2.
\end{align*}
Using this and maximum principle comparison against the ODE $\frac{d f}{dt} = \tfrac{1}{2n} f^2$ items (2) and (3) follow directly
\end{proof}

\section{Global existence results}

In this section we address global existence results for HKT Ricci flow.  We first formulate an existence conjecture based on the positive cone for the hyperHermitian $(2,0)$-forms in  quaternionic Bott-Chern cohomology, in analogy with the Tian-Zhang theorem for K\"ahler-Ricci flow \cite{TianZhang}.  We show that in this positive cone the flow reduces to a parabolic quaternionic Monge-Amp\`ere type equation, and use this to prove a conditional resolution of the cone conjecture (Theorem \ref{t:coneconjthm}).  We then prove the global existence on four-manifolds (Theorem \ref{t:4dthm}).

\subsection{Maximal existence time conjecture and scalar reduction} \label{ss:scalarreduction}

In the following, given an hyperHermitian structure $(I,J,K,g)$, we will set $\Phi(g):=-\partial_J \alpha$.
\begin{defn} \label{d:positivecone} Let $(M^{4n}, I, J, K)$ be a hypercomplex manifold.  Define the \emph{HKT cone} by
\begin{align*}
    \mathcal P = \{ [\alpha]_{qBC} \in H^{2,0}_{qBC}\ |\ \exists\ \Omega \in [\alpha]_{qBC},\ \Omega > 0 \}.
\end{align*}
\end{defn}

\begin{conj} \label{c:coneconj} Let $(M^{4n}, I, J, K, g_0, \Omega_0)$ be a compact HKT manifold.  Let
\begin{align*}
    \tau^* := \sup\ \{ t \geq 0\ |\ [\Omega_0]_{qBC} - t c_1^{qBC} \in \mathcal P \}
\end{align*}
The maximal smooth solution to (\ref{f:HKTflow}) exists on $[0, \tau^*)$.
\end{conj}

We show that this conjecture reduces to a parabolic quaternionic Monge-Amp\`ere flow against a moving background.  This is a standard argument adapted from \cite{TianZhang}.  Fix an initial HKT metric $\Omega_0$, which determines a relevant $\tau^*$, and choose $\tau < \tau^*$.  By definition of $\tau^*$ we know that
\begin{align*}
    [\Omega_0]_{qBC} - \tau c_1^{qBC} \in \mathcal P.
\end{align*}
Thus there exists an HKT metric $\til{\Omega}_{\tau} \in [\Omega_0] - \tau c_1^{qBC}$. Since $c_1^{qBC}=[-\partial_J \alpha_0]$ and $\Phi(\Omega_0)=-\partial_J \alpha_0$, $c_1^{qBC}=[\Phi(\Omega_0)]$. Therefore there exists $f \in C^{\infty}(M)$ such that
\begin{align*}
    \til{\Omega}_{\tau} = \Omega_0 - \tau \Phi(\Omega_0) + \del \del_J f = \Omega_0 - \tau \Phi(e^{\hat{f}} \Omega_0),
\end{align*}
where $\hat{f}=-\frac{f}{\tau n}$, and the last equality follows by Lemma \ref{l:transgression} .  Now we define a one-parameter family of background HKT metrics
\begin{align*}
    \hat{\Omega}_t := \frac{t}{\tau} \til{\Omega}_{\tau} + \frac{\tau - t}{\tau} \Omega_0,
\end{align*}
which is indeed positive definite for $0 \leq t \leq \tau$ by convexity of the space of positive $(2,0)$ forms.  An elementary computation shows that 
\begin{align*}
    [\hat{\Omega}_t]_{qBC} = [\Omega_0]_{qBC} - t c_1^{qBC},
\end{align*}
so that $\hat{\Omega}_t$ serves as a family of background metrics in the cohomology class of $\Omega_t$, where $\Omega_t$ is a solution of the HKT Ricci flow.
We now let $\varphi_t$ be the unique solution to
\begin{align} \label{f:PqMA}
    \dt \varphi = \log \frac{ \left( \hat{\Omega}_t - \del \del_J \varphi \right)^n}{ e^{n\hat{f}} \Omega_0^n}.
\end{align}
Using Lemma \ref{l:transgression} we observe 
\begin{align*}
    \dt (\hat{\Omega}_t - \del \del_J \phi_t) = -\Phi(e^{\hat{f}} \Omega_0) - \del \del_J \log \frac{ \left( \hat{\Omega}_t - \del \del_J \varphi \right)^n}{ e^{n\hat{f}} \Omega_0^n} = - \Phi( \hat{\Omega}_t - \del \del_J\varphi_t),
\end{align*}
hence $\hat{\Omega}_t - \del \del_J \varphi_t$ is the unique solution to (\ref{f:Omegaflow}) with initial condition $\Omega_0$.

\subsection{Proof of Theorem \ref{t:coneconjthm}}

To prove Theorem \ref{t:coneconjthm} we develop a priori estimates for the scalar reduction similar to the classic estimates of the parabolic complex Monge-Amp\`ere equation.  A further key role is played by a Schwarz lemma style estimate, relying on the Hermitian curvature flow interpretation of HKT Ricci flow.

For all of the estimates to follow we fix the setup of the previous subsection.  In particular we assume the existence of a $\tau < \tau^*$, a moving background $\hat{\Omega}_t$, and a corresponding unique solution of (\ref{f:PqMA}) with initial data $\phi_0=0$, which a priori exists smoothly for times $t$ in some subinterval of $[0, \tau]$.  In the following: $\Omega_{\phi}:= \hat \Omega-\partial \partial_J \phi$. 

\begin{prop} \label{p:C0estimate2} There exists a constant $C = C(\Omega_0, \tau)$ such that
\begin{align*}
    \sup_{M \times \{t\}} \brs{\varphi} \leq C t.
\end{align*}
\begin{proof} This is a standard maximum principle argument using equation (\ref{f:PqMA}) directly.
\end{proof}
\end{prop}

\begin{lem} \label{l:phidotev} One has
\begin{align*}
    \square \varphi =&\ \dot{\varphi} - n + \tr_{\omega_{\varphi}} \hat{\omega}\\
    \square \dot{\varphi} =&\  \tr_{\omega_{\varphi}} \Psi,
\end{align*}
where $\Psi$ is a fixed background tensor depending on $\Omega, \tau$.
\begin{proof}
The first item is a tautology after observing that
\begin{align*}
\Delta_{\omega_{\varphi}} \varphi = -\tr_{\Omega_{\varphi}} \del \del_J \varphi = \tr_{\Omega_{\varphi}} \left( \Omega_{\varphi} - \hat{\Omega} \right) = n -  \tr_{\Omega_{\varphi}} \hat{\Omega}.
\end{align*}
    Taking the time derivative of (\ref{f:PqMA}) and a standard computation yield
    \begin{align*}
        \dt \dt \varphi =&\ \tr_{\Omega_{\varphi}} \left( \tfrac{1}{\tau} \left( \til{\Omega}_{\tau} - \Omega_0 \right) - \del \del_J \dot{\varphi} \right) = \Delta_{\omega_{\varphi}} \dot{\varphi} + \tr_{\omega_{\varphi}} \Psi,
    \end{align*}
    where $\Psi$ is defined by the equality.
\end{proof}
\end{lem}

\begin{prop} \label{p:phidotestimate} For any smooth existence time $t$ there exists a constant $C = C(\Omega_0, t)$ such that
\begin{align*}
    \sup_{M \times \{t\}} \brs{\dot{\varphi}} \leq C.
\end{align*}
\begin{proof} Let $F := \dot{\varphi} - A \varphi$, where $A > 0$ is a constant to be determined.  Using Lemma \ref{l:phidotev} we obtain
\begin{align*}
    \square F =&\ \tr_{\omega_{\varphi}} \Psi - A \left( \dot{\varphi} - n + \tr_{\omega_{\varphi}} \hat{\omega} \right)\\
    =&\ \tr_{\omega_{\varphi}} \left(\Psi - A \hat{\omega} \right) - A \left( F + A \varphi \right) + A n\\
    \leq&\ - A F + C,
\end{align*}
where the last line follows by choosing $A$ large with respect to $\Psi$ and the estimate of Proposition \ref{p:C0estimate2}.  We obtain a time-dependent a priori upper bound for $F$ from the maximum principle, from which an upper bound for $\dot{\varphi}$ follows from Proposition \ref{p:C0estimate2}.  The argument for the lower bound on $\dot{\varphi}$ is analogous.
\end{proof}
\end{prop}

\begin{lem} \label{l:schwarz} Let $(M^{4n}, g_t, I, J, K)$ be a solution to HKT Ricci flow, and fix $\hat{g}$ a background Hermitian metric on $M$.  Then
\begin{align*}
    \square \tr_{\hat{\omega}} \omega_{\varphi} =&\ - \brs{\Upsilon(\omega_{\varphi},\hat{\omega})}^2_{\omega_{\varphi}^{-1},\hat{\omega}^{-1},\omega_{\varphi}} + \tr_{\hat{\omega}} Q - \omega_{\varphi}^{-1} \star \omega_{\varphi} \star R^C_{\hat{\omega}},\\
    \square \log \tr_{\hat{\omega}} \omega_{\varphi} \leq&\ C \brs{T}^2 + C \tr_{\omega_{\varphi}} \hat{\omega},
\end{align*}
where $\Upsilon(\omega_{\varphi}, \hat{\omega}) =\N^C_{\omega_{\varphi}} - \N^C_{\hat{\omega}}$.
\begin{proof} Using the expression of HKT Ricci flow as Hermitian curvature flow from Theorem \ref{t:mainthm1}, the evolution equation follows from a standard Schwarz lemma computation.  This was done for pluriclosed flow in \cite[Lemma 6.8]{ASnondeg}.  The computation there holds for arbitrary Hermitian curvature flows, yielding the claimed evolution.  The inequality for $\log \tr_{\hat{\omega}} \omega_{\varphi}$ follows as in \cite[Lemma 6.9]{ASnondeg}, after observing an elementary estimate $(\tr_{\hat{\omega}} \omega_{\varphi})^{-1} \tr_{\hat{\omega}} Q \leq C \brs{T}^2_{\omega_{\varphi}}$ for some constant $C$.
\end{proof}
\end{lem}

\begin{thm} \label{t:coneconjthmbulk} (cf. Theorem \ref{t:coneconjthm}) Let $(M^{4n}, I, J, K, g_0)$ be a compact HKT manifold.  Suppose the solution to HKT Ricci flow with this initial data exists on $[0, \tau)$ where $\tau < \tau^*$ and either
\begin{enumerate}[label={(\arabic*)}]
    \item $\sup_{M \times [0, \tau)} \tr_{g_{0}} g < \infty$,
    \item $\sup_{M \times [0, \tau)} \brs{T} < \infty$.
\end{enumerate}
Then the flow extends smoothly past time $\tau$.
\end{thm}

\begin{proof} Since $\tau < \tau^*$ we may set up the scalar reduction as above, and aim to prove uniform a priori estimates on $[0, \tau)$.  Propositions \ref{p:C0estimate2} and \ref{p:phidotestimate} yield estimates for $\varphi$ and $\dot{\varphi}$, hence also on the volume form.  Thus, to establish uniform parabolicity of the equation it suffices to obtain a uniform upper bound.  In case (1) of the theorem this bound is assumed.  In case (2) we apply the maximum to $\log \tr_{\hat{\omega}} \omega_{\varphi} - A \varphi$, exploiting the estimate of Lemma \ref{l:schwarz} and the assumed bound on torsion.

Having established the uniform equivalence of the metrics and a uniform bound for the potential, it remains to show the higher order estimates.  The key $C^{2,\alpha}$ estimate follows using the general Evans-Krylov type estimate developed by Chu \cite{chu2016c, chu2020parabolic} (cf. \cite[Theorem 3.3]{BGV}).  All higher order estimates follow then via Schauder theory, finishing the proof by a standard argument.
\end{proof}

\subsection{Four dimensional HKT Ricci flow} \label{sub:4-dim}

In this subsection we prove Theorem \ref{t:4dthm}.  In dimension four, every compact HKT manifold $(M,I,J,K,g)$ is conformal to one of either a flat torus, hyperK\"ahler $K3$ surface, or a a quaternionic Hopf surface equipped with its standard locally conformally flat metric \cite{Bo}.  Every hyperKähler metric is clearly HKT-Einstein with Einstein constant equal to zero.  The standard locally conformally flat metric on a Hopf surface is also HKT-Einstein with HKT-Einstein constant equals to $1$, which follows by a direct computation.

We start with the following standard lemma to fix notation and conventions.
\begin{lem} \label{l:conf}
Let $(M^4,I,J,K,\bar g)$ be a HKT manifold, and let $(I,J,K,e^u \bar g)$ be a HKT conformal structure. Then
\begin{equation} \label{eqn:rho_conf}
    \rho(e^u \bar \omega)=\rho(\bar \omega)+dd^cu.
\end{equation}
and
\begin{equation} 
    s(e^u \bar \omega)=e^{-u}\left(s(\bar \omega)-2 \Delta_{\bar \omega}u \right).
\end{equation}
\end{lem}
\begin{proof}
This is a consequence of the fact that $\theta_{e^u\bar g}=\theta_{\bar g}+du$ and $\rho (e^u \bar \omega)=dI\theta_{e^u \bar \omega}$.
\end{proof}

\begin{lem} \label{l:phi4d}
Let $(M^4,I,J,K, \bar g)$ be a HKT manifold and consider $(I,J,K,e^u \bar g)$ a conformal HKT structure. Then 
\[
\Phi(e^u \bar g)=\Phi(\bar g)-(\Delta_{\bar \omega}u) \, \bar \omega.
\]
\end{lem}
\begin{proof}
It follows by Lemma \ref{l:conf} that 
\[
\rho(e^u \bar \omega)=\rho(\bar \omega)+dd^c_I u.
\]
By definition,
\[
\Phi(e^u \bar g)=\Phi(\bar g)+\frac{dd^c_I u - J dd^c_I u}{2}.
\]
$\frac{dd^c_I u - J dd^c_I u}{2}$ is a multiple of $\bar \omega$: this follows since it is a Salamon $2$-form on a $4$-dimensional manifold. Hence 
\[
\frac{dd^c_I u- J dd^c_I u}{2}= -(\Delta_{ \bar \omega}u) \, \bar \omega,
\]
concluding the proof.
\end{proof}

We first observe that the flow naturally reduces to a conformal flow, which up to scale is the Chern–Yamabe flow considered in \cite{Calamai} (see also \cite{ACS,LM}). 

\begin{lem} \label{l:4dconformal}
Let $(M^4, g_t, I, J, K)$ be a solution to (\ref{f:HKTflow}).  Then there exists an HKT-Einstein metric $\bar{g}$ on $M$ with constant $\lambda \geq 0$ and $u_t$ a solution of
\begin{align} \label{f:conformalflow}
    \dt u =&\ e^{-u} \left(  \Delta_{\bar{\omega}} u - \lambda \right)
\end{align}
such that $g_t = e^{u_t} \bar{g}$.  In the case $\lambda > 0$ the solution to the normalized flow (\ref{eq:normalizedHKTflow}) is described by $g_t = e^{u_t} \bar{g}$ where
\begin{align} \label{f:normalizedconformalflow}
    \dt u =&\ e^{-u} \left( \Delta_{\bar{\omega}} u - \lambda \right) + \lambda.
\end{align}

\begin{proof} As discussed above, $M^4$ is either a hyperK\"ahler manifold or a quaternionic Hopf surface, and moreover there exists $u_0 \in C^{\infty}(M)$ such that $g = e^{u_0} \bar{g}$ where $\bar{g} = g_{\mathrm{HK}}$ or $g_{\mathrm{Hopf}}$, which are HKT-Einstein with constant $\lambda = 0$ or $\lambda > 0$ respectively. Using Lemma \ref{l:phi4d} we know that
\begin{align*}
    \Phi(e^{u} \bar{g}) = - \Delta_{\bar{\omega}} u \, \bar{\omega} + \Phi(\bar{g}) = -  \left( e^{-u} \Delta_{\bar{\omega}} u \right) e^u  \bar{\omega} + \lambda \bar{\omega} = e^{-u} \left( -  \Delta_{\bar{\omega}} u + \lambda \right) e^u \bar{\omega}.
\end{align*}
It follows from a straightforward computation that if $u_t$ is the unique solution to (\ref{f:conformalflow}) with initial condition $u_0$ then $g_t = e^{u_t} \bar{g}$, as claimed.  For the case of the normalized flow, we observe that we can express
\begin{align*}
    \Phi(e^{u} \bar{g}) - \lambda e^{u} \bar{\omega} = \left[ e^{-u} \left( - \Delta_{\bar{\omega}} u + \lambda \right) - \lambda \right] e^u \bar{\omega},
\end{align*}
and the claim follows as above.
\end{proof}
\end{lem}

\begin{proof}[Proof of Theorem \ref{t:4dthm}] The case of hyperK\"ahler background of course follows corresponds to solving equation (\ref{f:conformalflow}) with $\lambda = 0$.  An elementary maximum principle argument shows that $u$ has a priori upper and lower bounds which are uniform in time.  The higher regularity can be obtained as in the proof of Theorem \ref{t:coneconjthm}, or directly from the Krylov-Safonov regularity theory.  The convergence follows from a straightforward adaptation of the method of Li-Yau \cite{CaoKRF,li1986parabolic}.

In the case of a Hopf surface background, we apply the maximum principle to \eqref{f:normalizedconformalflow} to show time-dependent upper and lower bounds for $u$.  The higher regularity is dealt with as in the hyperK\"ahler case.
\end{proof}

\section{HKT-Einstein metrics} \label{s:HKTE}

In this section we study the structure of compact HKT-Einstein metrics, centering on a deeper understanding of the relationship of HKT-Einstein metrics to strong HKT metrics, i.e. those with $d H = 0$.  We give distinct characterizations of HKT-Einstein metrics in the cases of $n = 4$ and $n > 4$, show that balanced HKT metrics are precisely the HKT-Einstein metrics with vanishing scalar curvature, and exhibit a Bochner vanishing result.  Then we prove Theorem \ref{t:sHKTE} which characterizes which of the known strong HKT metrics (coming from bi-invariant metrics on Lie groups) are HKT-Einstein.  Finally we prove Theorem \ref{t:4dnonuniquenessinto}, which constructs a divergent sequence of (non-strong) HKT-Einstein metrics on quaternionic Hopf surfaces.

\subsection{Fundamental structure}

\begin{lem} \label{l:constant_chern_scalar}
Let $(M^4,I,J,K)$ be a compact hypercomplex $4$ manifold, and let $\bar g=g_{\mathrm{HK}}$ or $g_{\mathrm{Hopf}}$. Then a HKT metric $g=e^u \bar g$ is HKT-Einstein if and only if it has constant (positive) Chern scalar curvature. 
\end{lem}
\begin{proof}
The direct implication is obvious. Let us prove the converse. Assume that $g$ has constant (positive) Chern scalar curvature $s$. By Lemma \ref{l:conf}, 
\begin{equation*} 
    s= e^{-u}(2 \lambda -2 \Delta_{\bar \omega}u), 
\end{equation*}
where $\lambda=0$ if $g= e^u g_{\mathrm{HK}}$ and $\lambda=1$ if $g= e^u g_{\mathrm{Hopf}}$.  On the other hand, by Lemma \ref{l:phi4d}, 
\[
\Phi(g)= \lambda \bar \omega-(\Delta_{\bar \omega}u) \bar \omega=\lambda \bar \omega + \tfrac s2e^u  \bar \omega - \lambda \bar \omega= \tfrac s2e^u \bar \omega= \tfrac s2 \omega.
\]
\end{proof}
\begin{rmk}
The preceding proof also shows that if $ g = e^{u} g_{\mathrm{HK}}$, then the unique (up to rescaling) HKT-Einstein metric in the conformal class is the hyperK\"ahler metric itself. Indeed, setting $\lambda=0$ in \eqref{eqn:hkteinstein4d} yields
\[
\Delta_{\bar \omega} u = -e^{u} s.
\]
By the maximum principle, $u$ must be constant and $s=0$. For the Hopf metric, the uniqueness argument fails. Since the Hopf surface is not K\"ahler, any HKT-Einstein metric must have non-zero HKT-Einstein constant. Up to absorbing the rescaling in the conformal factor, we may assume that it is equal to $1$. By the above lemma, a HKT metric $g=e^u g_{\mathrm{Hopf}}$  is HKT-Einstein if and only if 
\begin{equation}
\label{eqn:hkteinstein4d}
    2= e^{-u}(2 \lambda -2 \Delta_{\bar \omega}u).
\end{equation}
We will show in Section \ref{s:Hopf} that we have actually infinite many solutions of \eqref{eqn:hkteinstein4d}. 
\end{rmk}

The following lemma is well-known to experts. We will use it to obtain the characterization of the HKT-Einstein condition in quaternionic dimension at least two.
\begin{lem} \label{lemma:conformal}
Let $(M^{4n},I,J,K)$ be an hypercomplex manifold, and let $(I,J,K,g)$ be an HKT structure. For any  non-constant function $u$, $(I,J,K,e^ug)$ is HKT if and only if $n=1$. 
\end{lem}
\begin{proof}
The hyperHermitian structure $(I,J,K,e^ug)$ is HKT if and only if 
\[
d^c_I (e^u\omega)=d^c_J (e^u\omega_J)=d^c_K (e^u\omega_K).
\]
Since 
\[
d^c_L (e^u \omega_L)=d^c_L(e^u) \wedge \omega_L + e^u d^c_L\omega_L,
\]
the HKT condition reduces to proving that 
\[
d^c_I(e^u) \wedge \omega = d^c_J(e^u) \wedge \omega_J = d^c_K(e^u) \wedge \omega_K.
\]

Fix a point $p \in M$ in which $de^u \neq 0$ and set $e^1 = \frac{de^u}{|de^u|}$. Complete $e^1$ to an orthonormal frame of $T_pM$ of the form
$
\{e^1, Ie^1, Je^1, Ke^1, e^i, Ie^i, Je^i, Ke^i\}_{i=2}^n.
$
With respect to this frame, we compute
\begin{align*}
d^c_I(e^u) \wedge \omega &= |de^u| Ie^1 \wedge \bigl(Ie^1 \wedge e^1 + Ke^1 \wedge Je^1 + \sum_{i=2}^n Ie^i \wedge e^i + Ke^i \wedge Je^i\bigr),\\
d^c_J(e^u) \wedge \omega_J &= |de^u| Je^1 \wedge \bigl(Je^1 \wedge e^1 + Ie^1 \wedge Ke^1 + \sum_{i=2}^n Je^i \wedge e^i + Ie^i \wedge Ke^i\bigr).
\end{align*}
It follows immediately that these two expressions are equal if and only if $n=1$. Indeed, if $n \geq 2$, we may evaluate 
\[
d^c_I(e^u) \wedge \omega (Ie_1, e_2, Ie_2) \neq 0,
\]
while 
\[
d^c_J(e^u) \wedge \omega_J (Ie_1, e_2, Ie_2) = 0.
\]
\end{proof}
\begin{prop}
Let $(M^{4n}, I, J, K)$ be a hypercomplex manifold with $n \ge 2$. An HKT structure $(I,J,K,g)$ is HKT-Einstein if and only if $\mathring{\Phi}(g) = 0$, where $\mathring{\Phi}(g)$ denotes the trace-free part of $\Phi(g)$ with respect to $\omega$.
\end{prop}

\begin{proof}
The forward direction is clear. We prove the converse. Assume that $\mathring{\Phi}(g) = 0$. Then $\Phi(g) = \lambda \omega$ for some function $\lambda \in C^\infty(M)$. We claim that $\lambda$ is constant.  Since $d\Phi(g) \in  \bigcap\limits_{L \in S^2} \big( \Omega^{2,1}_L \oplus \Omega^{1,2}_L \bigr)$, we have
\begin{align*}
d(\lambda \omega) \in \bigcap_{L \in S^2} \bigl( \Omega^{2,1}_L \oplus \Omega^{1,2}_L \bigr).
\end{align*}
If $\lambda = 0$, the claim follows immediately. Thus we assume $\lambda \neq 0$ and consider a point $p$ such that $\lambda(p) \neq 0$. On a neighborhood $U$ of $p$ where $\lambda$ does not vanish, we may assume without loss that $\lambda > 0$.  By Definition~\ref{defn:HKT}, $(I,J,K,\lambda g)$ is a conformal HKT structure on $U$, which forces $\lambda$ to be constant on $U$ by Lemma~\ref{lemma:conformal}. Since $\lambda$ is smooth and $M$ is connected, $\lambda$ must be constant everywhere.
\end{proof}
It turns out that the case $\lambda = 0$ can be characterized completely in the compact case.  We recall that a Hermitian metric is called \emph{balanced} if $\theta = 0$.
We show that zero constant HKT-Einstein metrics are precisely balanced HKT metrics.  This is shown in \cite[Proposition 6.4]{fusi2026} for $n > 2$, and we give a slightly simplified proof here for convenience.
\begin{prop}
Let $(I,J,K,g)$ be an HKT structure on a compact hypercomplex manifold $(M^{4n},I,J,K)$. Then $
\Phi(g)=0 $
if and only if the HKT structure is balanced. 
\end{prop}
\begin{proof}

If $(I,J,K,g)$ is balanced, then $\rho$ vanishes and the equality above is satisfied. We now prove that $\rho = J \rho$ implies balanced.  Since $\rho = J \rho$, $s = 0$. However, by equation~\eqref{eqn:rhoChern} and Lemma \ref{l:trdItheta},
\[
0 = 2s = \sum_{i=1}^{4n} d I\theta (Ie_i, e_i) = 2|\theta|^2 + 2 d^\star \theta.
\]
Integrating, we obtain
\[
0 = \int |\theta|^2 + d^\star \theta \, \mathrm{vol} = \int  |\theta|^2 \, \mathrm{vol},
\]
which implies $\theta = 0$; hence the HKT structure is balanced.
\end{proof}

\begin{rmk} We point out that balanced HKT manifolds are, in general, not hyperKähler. A counterexample is given by hypercomplex nilmanifolds endowed with an abelian hypercomplex structure: in this case, any (left invariant) hyperHermitian metric is HKT \cite{DF,BDV}; however, nilmanifolds are Kähler only if they are tori.
\end{rmk}

The following Proposition provides a cohomological characterization of HKT-Einstein manifolds according to the sign of the HKT-Einstein constant:
\begin{prop}
Let $(M,I,J,K,g)$ a compact HKT-Einstein manifold. Then either $g$ is balanced or $h_L^{2n,0}=0$, where $h_L^{2n,0}$ is the Dolbeault-Hodge number of the complex structure $L$.
\end{prop}
\begin{proof}
Assume that $(I,J,K,g)$ is non balanced. Then for any Hermitian structure $(L,g)$, the Chern scalar curvature is strictly positive (see Equation \eqref{eqn:schern} and \eqref{eqn:chern_scalar}). Then, for any holomorphic $(2n,0)$ form $\eta$, 
the term $g( \eta, S_L\circ
\eta)= s \,  |\eta|^2$, where
\begin{align*}
\left( S \circ \eta \right)_{i_1 \dots i_p} =&\ \frac{1}{p!} \sum_{j = 1}^p
S_{i_j}^k \eta_{i_1 \dots i_{j-1} k i_{j+1} \dots i_p}.
\end{align*} 
We use the Bochner formula for $\eta$, that, in our case, reads as (see \cite{KobayashiWu1970}, \cite{Bochner}):
\begin{align} \label{boch}
\Delta_{\omega_{L}} \brs{\eta}^2 =&\ \brs{\N^C_L \eta}^2 + \brs{\bar{\N^C_L} \eta}^2 + s \,  |\eta|^2,
\end{align}
where $\N^C_L$ is the Chern connection.  Let $p$ be the point where $|\eta_p|^2$ attains its maximum.  Then the maximum principle applied to \eqref{boch}, gives
\[
0 \ge \Delta_{\omega_{L}} |\eta|^2 \ge 0,
\]
implying that $s |\eta_p|^2=0$. Since $s$ is constant and strictly positive, this forces $|\eta_p|^2=0$ and hence $\eta=0$.  
\end{proof}

\subsection{Strong HKT-Einstein structures} \label{s:strongHKTE}

It follows from the discussion in \S \ref{sub:4-dim} that $4$-dimensional strong HKT-Einstein structures are either hyperK\"ahler or the standard metrics on Hopf surfaces. 
In higher dimensions, the general theory remains rather unexplored. In the homogeneous setting, examples of HKT-Einstein metrics have been constructed \cite{fusi2026,BedulliMarcocci2026}; however, as we show below, only a few of them are strong HKT.  In the simply connected case, every compact $8$-dimensional strong HKT manifold is HKT-Einstein \cite[Proposition 7.1]{brienza2026}. Currently, the only known (non-hyperK\"ahler) example of a simply connected compact strong HKT manifold in dimension $8$ is $\SU(3)$.  It is well known that a compact Lie group with a bi-invariant metric and a compatible left-invariant hypercomplex structure yields the standard examples of strong HKT manifolds, which are actually Bismut flat. In \cite{DT,BGP}, it was shown that any left-invariant hypercomplex structure on a compact Lie group arises via Joyce's construction \cite{Joy}. Hence, we may focus on Joyce hypercomplex manifolds to better understand the relation between Bismut flat strong HKT structures and HKT-Einstein metrics.  A characterization of such spaces is mentioned in \cite[Remark 5.9]{BedulliMarcocci2026}, and we provide a direct proof of the classification here.

\medskip

To begin we recall the Joyce construction associated to a compact semisimple Lie group. We exclude the case when $G$ is abelian, indeed in such case Joyce hypercomplex manifolds are just tori $\mathbb{T}^{4k}$, which are hyperK\"ahler, and hence strong HKT-Einstein.  Set $r=\operatorname{rank}{G}$. Let $H$ be a maximal torus in $G$, and let $\mathfrak{h}$ and $\mathfrak{g}$ denote their Lie algebras. Choose an ordered root system $\Delta$ with respect to $\mathfrak{h}_{\mathbb{C}}$ and fix a highest positive root $\alpha_1$. Let $\mathfrak{d}_1$ be the $\mathfrak{su}(2)$-subalgebra of $\mathfrak{g}$ whose complexification is isomorphic to the $\mathfrak{sl}(2,\mathbb{C})$-subalgebra $[\mathfrak{g}_{\alpha_{1}},\mathfrak{g}_{-\alpha_{1}}]\oplus\mathfrak{g}_{\alpha_1}\oplus\mathfrak{g}_{-\alpha_1}$, where $\mathfrak{g}_{\pm\alpha_1}$ are the root spaces for $\pm\alpha_1$. Let $\mathfrak{c}_1$ denote the centraliser of $\mathfrak{d}_1$. There exists a real subspace $\mathfrak{f}_1$ of dimension $4d_1$ ($d_1\in \mathbb{N}_0$) such that $\mathfrak{g}=\mathfrak{c}_1\oplus\mathfrak{d}_1\oplus \mathfrak{f}_1$, where
\[
\mathfrak{f}_1=\mathfrak{g} \cap \bigoplus_{\substack{\alpha_1 \neq \alpha > 0 \\ ( \alpha, \alpha_1 ) \neq 0}} \mathfrak{g}_{\alpha}\oplus\mathfrak{g}_{-\alpha}.
\]
The subalgebra $\mathfrak{c}_1$ is the direct sum of an abelian Lie algebra and a semisimple Lie algebra $\mathfrak{g}'$. If $\mathfrak{g}'$ is nontrivial, Joyce reapplies the above decomposition to $\mathfrak{g}'$, which has an induced root system $\Delta' = \{\beta\in\Delta \mid (\alpha_1,\beta)=0\}$. Recursively, one obtains a decomposition of $\mathfrak{g}$ of the form \cite[Lemma 4.1]{Joy}:
\begin{equation}\label{eqn:Joycedec}
\mathfrak{g}=\mathfrak{b}\oplus \bigoplus_{i=1}^m\mathfrak{d}_i\oplus \bigoplus_{i=1}^m\mathfrak{f}_i\,,
\end{equation}
where
\begin{enumerate} [label={(\arabic*)}]
\item $\mathfrak{b}$ is abelian of dimension $r-m$, 
\item $\mathfrak{d}_i\cong \mathfrak{su}(2)$ for each $i=1,\dots,m$, 
\item $\mathfrak{f}_i$ are (possibly trivial) subspaces.
\end{enumerate} 
The following commutation relations hold:
\begin{enumerate}[label=(J\arabic*),ref=J\arabic*]
\item $[\mathfrak{d}_i,\mathfrak{b}]=0$ for all $i$;
\item $[\mathfrak{d}_i,\mathfrak{d}_j]=0$ for $i\neq j$;
\item $[\mathfrak{d}_i,\mathfrak{f}_j]=0$ for $i<j$;
\item $[\mathfrak{d}_i,\mathfrak{f}_i]\subseteq \mathfrak{f}_i$, and the action of $\mathfrak{d}_i$ on $\mathfrak{f}_i$ is isomorphic to a direct sum of copies of the standard $\mathfrak{su}(2)$-action on $\mathbb{C}^2$.
\end{enumerate}
We call \eqref{eqn:Joycedec} a \emph{Joyce decomposition}. Although the construction depends on the choice of Cartan subalgebra, positive roots, and highest root, different choices yield isomorphic decompositions via an inner automorphism of $\mathfrak{g}$. Thus we may speak unambiguously of \emph{the} Joyce decomposition.
\medskip

Now we define hypercomplex structures using the Joyce decomposition.  Let $\mathbb{T}^{2m-r}\cong \mathrm{U}(1)^{2m-r}$ be a $(2m-r)$-dimensional torus and set $\ell:=2m-r$. Then $\ell \mathfrak{u}(1)\oplus \mathfrak{b}\cong \mathbb{R}^m$. Fix a basis $\mathcal{B}=(e_1^1,\dots,e_1^m)$ of $\ell \mathfrak{u}(1)\oplus \mathfrak{b}$. The Joyce hypercomplex structure on $\mathbb{T}^{\ell}\times G$ depends solely on $\mathcal{B}$ and it is defined on each layer of \eqref{eqn:Joycedec} as follows:
\begin{itemize}
    \item For each $i=1,\dots,m$, fix a basis $(e_2^i,e_3^i,e_4^i)$ of $\mathfrak{d}_i\cong\mathfrak{su}(2)$ satisfying
\[
[e_2^i,e_3^i]=2e^i_4,\quad [e_4^i,e^i_2]=2e^i_3,\quad [e^i_3,e^i_4]=2e^i_2.
\]
Then $(e_1^i,e_2^i,e_3^i,e_4^i)$ is a basis of $\mathfrak{h}_i:=\mathbb{R}\oplus\mathfrak{d}_i\cong\mathbb{R}\oplus\mathfrak{su}(2)$, identified with the quaternions. The hypercomplex structure on $\mathfrak{h}_i$ is given by:
\[
Ie^i_1=e^i_2,\; Ie^i_3=e^i_4,\; Je^i_1=e^i_3,\; Je^i_2=-e^i_4,\; Ke^i_1=e^i_4,\; Ke^i_2=e^i_3.
\]
\item The action of $I,J,K$ is extended to each $\mathfrak{f}_i$ via
\[
If=[e^i_2,f],\quad Jf=[e^i_3,f],\quad Kf=[e^i_4,f],\qquad f\in\mathfrak{f}_i.
\]
\end{itemize}
The complex structures $\{I,J,K\}$ are defined at the identity and extended by left translation to $\mathbb{T}^{\ell}\times G$. By \cite{Sam,Joy}, they are integrable and define a homogeneous hypercomplex structure, called a \emph{Joyce hypercomplex manifold}.

Let $g$ be any bi-invariant metric on $G$, which therefore coincides with a multiple of the negative of the Killing-Cartan form on each simple factor. The Joyce decomposition is $g$-orthogonal \cite{grantcharov2000}. Choose a standard basis $(e_2^j,e_3^j,e_4^j)$ of $\mathfrak{d}_j$ such that $g(e_2^j,e_2^j)=g(e_3^j,e_3^j)=g(e_4^j,e_4^j)=\lambda_j^2$. Take a basis $(e_1^1,\dots,e_1^\ell,e_1^{\ell+1},\dots,e_1^{m})$ of $\ell\mathfrak{u}(1)\oplus\mathfrak{b}\cong\mathbb{R}^m$, where $(e_1^1,\dots,e_1^\ell)$ is a basis of $\ell\mathfrak{u}(1)$ and $(e_1^{\ell+1},\dots,e_1^{m})$ is a $g$-orthogonal basis of $\mathfrak{b}$ with $g(e_1^{\ell+j},e_1^{\ell+j})=\lambda_{\ell+j}^2$ for $j=1,\dots,m-\ell$. Extend $g$ to a positive definite bilinear form on $\ell\mathfrak{u}(1)\oplus\mathfrak{g}$ by setting $g(e_1^{j},e_1^{j})=\lambda_j^2$ for $j=1,\dots,\ell$. By construction,   $g$ is a bi-invariant metric that is hyperhermitian with respect to the Joyce hypercomplex structure defined by $(e_1^1,\dots,e_1^\ell,e_1^{\ell+1},\dots,e_1^{m})$. Furthermore, it is easy to check that it is strong HKT. 
\begin{prop}[\cite{BFGV}, \cite{OP} Equation (64)] \label{p:hkteinstein}
Let $G$ be a compact semisimple Lie group and let $\mathbb{T}^{\ell}\times G$ be such that $(\mathbb{T}^{\ell}\times G,I,J,K)$ admits a bi-invariant strong HKT metric $g$ constructed as above. Then the Lee form of the HKT structure $(I,J,K,g)$ is given by
\[
\theta=\sum_{j=1}^m \frac{(\alpha_j, \alpha_j)}{2} \left(1+\dim_{\mathbb{H}}(\mathfrak{f}_j)\right)g(e_1^j,\cdot),
\]
where $\alpha_j$ is the highest root at the $j$-th layer of the Joyce decomposition.
\end{prop}

\begin{thm} \label{thm:StrongHKTEinstein}
Let $G$ be a simple Lie group and let $(\mathbb{T}^\ell \times G, I,J,K,g)$ be any strong HKT manifold constructed as above. Then $(I,J,K,g)$ is HKT-Einstein if and only if $\mathbb{T}^\ell \times G = S^1\times \mathrm{SU}(2)$ or $\mathrm{SU}(3)$.
\end{thm}

\begin{proof}
By Proposition \ref{p:hkteinstein}, the HKT-Einstein condition is
\begin{equation} \label{eqn:HKT-Joyce}
   -\sum_{j=1}^m \frac{(\alpha_j, \alpha_j)}{4} \left(1+\dim_{\mathbb{H}}(\mathfrak{f}_j)\right)g(e_2^j,[X,Y]-[JX,JY])=\lambda \omega(X,Y).  
\end{equation}
We evaluate equation \eqref{eqn:HKT-Joyce} on $(e_1^j,e_2^j)$ for each $j=1,\dots,m$. If $g$ is HKT-Einstein then
\begin{equation} \label{eqn:HKT_Einstein}
    \frac{(\alpha_j, \alpha_j)}{2} \left(1+\dim_{\mathbb{H}}(\mathfrak{f}_j)\right)=\frac{(\alpha_k, \alpha_k)}{2} \left(1+\dim_{\mathbb{H}}(\mathfrak{f}_k)\right)
\end{equation}
for all $j,k=1,\dots,m$. Using this necessary condition, we prove the claimed classification.

We first show that $S^1\times \mathrm{SU}(2)$ and $\mathrm{SU}(3)$ are both HKT-Einstein. The first claim has been already discussed, while for $\mathrm{SU}(3)$ the result follows from \cite[Proposition 7.1]{brienza2026}. We now exclude all remaining cases. In the following, we will always assume that $g$ is rescaled so that long roots have norm $2$.

We shall use the following facts. First, except for the Hopf surface $S^1\times \SU(2)$, we have $\dim_{\mathbb H}(\mathfrak f_1)>0$. Second, apart from the groups $\SU(2k+1)$ (i.e., type $A_{2k}$ with $k\ge 2$), every Joyce decomposition contains at least one trivial $\mathfrak{f}_i$ factor \cite{BFGV}.

Assume that $G \neq \SU(2)$. If all roots $\alpha_j$ have the same length, then the HKT-Einstein equation \eqref{eqn:HKT-Joyce} implies that each $\mathfrak f_j$ has the same dimension. If a trivial summand $\mathfrak f_j$ exists, this is impossible, since $\mathfrak f_1$ is non-trivial. 

Since types $A_n$ $(n \ge 3)$, $D_n$ $(n \ge 4)$, $E_6$, $E_7$, and $E_8$ are simply laced, all roots have the same length. In each corresponding simple Lie group, except for $\mathrm{SU}(2k+1)$, we may apply the argument above. For $\mathrm{SU}(2k+1)$ $k \ge 2$, a different argument is needed: one uses the fact that $\dim_{\mathbb{H}}(\mathfrak{f}_j) < \dim_{\mathbb{H}}(\mathfrak{f}_{j-1})$ to again rule out \eqref{eqn:HKT_Einstein} \cite{BFGV}.

For type $C_n$ ($n>1$), the successive highest roots are $\alpha_i=2e_i$, all long, hence their length is constant. Since a trivial summand exists, the same contradiction applies.

For type $B_n$ ($n\ge 3$), the roots $\alpha_1=e_1+e_2$ and $\alpha_2=e_1-e_2$ are both long, so $(\alpha_1,\alpha_1)=(\alpha_2,\alpha_2)$, while $\mathfrak f_2$ is trivial. Thus $\dim_{\mathbb H}(\mathfrak f_1)>0$ and $\dim_{\mathbb H}(\mathfrak f_2)=0$, giving the desired contradiction.

For type $F_4$, the roots $\alpha_1,\alpha_2,\alpha_3,\alpha_4$ are all long, so their length is constant. Since a trivial $\mathfrak f_j$ factor exists, we again get a contradiction.

For type $G_2$ the argument is slightly different. The highest root $\alpha_1$ is long and the next highest root $\alpha_2$ is short, with $(\alpha_2,\alpha_2)=\frac{2}{3}$. Moreover, $\dim_{\mathbb H}(\mathfrak f_1)=2$ and $\dim_{\mathbb H}(\mathfrak f_2)=0$. Substituting into \eqref{eqn:HKT-Joyce} gives $3=1/3$, absurd.

This completes the proof.
\end{proof}
We turn our attention to the semisimple case.
\begin{thm} \label{t:sHKTE}
Let $G$ be a non-trivial compact semisimple Lie group and let $(\mathbb{T}^\ell \times G, I,J,K,g)$ be any strong HKT manifold constructed as above. Then $(I,J,K,g)$ is HKT-Einstein if and only if
\[
\mathbb{T}^\ell \times G \cong \bigl( S^1 \times \mathrm{SU}(2) \bigr) \times \dots \times \bigl( S^1 \times \mathrm{SU}(2) \bigr) \times \mathrm{SU}(3) \times \dots \times \mathrm{SU}(3).
\]
\end{thm}
\begin{proof} We give the proof in the case $G = G_1 \times G_2$.  We emphasize that while the Joyce decomposition naturally splits, a priori our metric is determined by a basis $\mathcal B$ of the abelian part, and in particular the appended torus factors need not be orthogonal.
The Joyce decomposition of $\mathfrak{g}$ is then given by
\[
\mathfrak{g} = \biggl( \mathfrak{b}_1 \oplus \bigoplus_{i=1}^{m_1} \mathfrak{d}_i \oplus \bigoplus_{i=1}^{m_1} \mathfrak{f}_i \biggr) \oplus \biggl( \mathfrak{b}_2 \oplus \bigoplus_{i=1}^{m_2} \mathfrak{d}_{m_1+i} \oplus \bigoplus_{i=1}^{m_2} \mathfrak{f}_{m_1+i} \biggr),
\]
and the corresponding $\mathbb{T}^\ell \times G$ is given by 
\[
\mathbb{T}^\ell \times G = \mathbb{T}^{\ell_1} \times \mathbb{T}^{\ell_2} \times G_1 \times G_2,
\]
where $\ell = \ell_1 + \ell_2$.

Let $g$ be any bi-invariant metric on $G$, which therefore coincides with $g_1 \times g_2$, where $g_i$ is a multiple of the negative Killing-Cartan form on $\mathfrak{g}_i$. For each $j = 1,\dots,m_1+m_2$, choose a standard basis $(e_2^j,e_3^j,e_4^j)$ of $\mathfrak{d}_j$ such that
\[
g(e_2^j,e_2^j)=g(e_3^j,e_3^j)=g(e_4^j,e_4^j)=\lambda_j^2.
\]
Take a basis
\[
\mathcal{B}=\bigl( e_1^1,\dots,e_1^{\ell_1},\; e_1^{\prime 1},\dots,e_1^{\prime \ell_2},\; e_1^{\ell_1+1},\dots,e_1^{m_1},\; e_1^{\prime \ell_2+1},\dots,e_1^{\prime m_2} \bigr)
\]
of $(\ell_1+\ell_2)\mathfrak{u}(1) \oplus \mathfrak{b} \cong \mathbb{R}^{m_1+m_2}$, where:
\begin{itemize}
    \item $(e_1^{\ell_1+1},\dots,e_1^{m_1})$ is a basis of $\mathfrak{b}_1$ that is $g_1$-orthogonal with
    \[
    g_1(e_1^{\ell_1+j},e_1^{\ell_1+j}) = \lambda_{\ell_1+j}^2 \qquad \text{for } j = 1,\dots,m_1-\ell_1;
    \]
    \item $(e_1^{\prime \ell_2+1},\dots,e_1^{\prime m_2})$ is a basis of $\mathfrak{b}_2$ that is $g_2$-orthogonal with
    \[
    g_2(e_1^{\prime \ell_2+j},e_1^{\prime \ell_2+j}) = \lambda_{\ell_2+j}^2 \qquad \text{for } j = 1,\dots,m_2-\ell_2.
    \]
\end{itemize}
Extend $g$ to a positive definite bilinear form on $\ell_1\mathfrak{u}(1) \oplus \ell_2\mathfrak{u}(1) \oplus \mathfrak{b}$ by setting
\[
g(e_1^{j},e_1^{j}) = \lambda_j^2 \quad \text{for } j = 1,\dots,\ell_1,
\qquad
g(e_1^{\prime j},e_1^{\prime j}) = \lambda_{m_1+j}^2 \quad \text{for } j = 1,\dots,\ell_2.
\]
Clearly, $g$ is hyperhermitian with respect to the Joyce hypercomplex structure defined by $\mathcal{B}$, as already observed. 
Consequently,
\[
(\ell_1+\ell_2)\mathfrak{u}(1) \oplus \mathfrak{g}_1 \oplus \mathfrak{g}_2
= \Bigl( \mathbb{R}\langle e_1^1,\dots,e_1^{\ell_1} \rangle \oplus \mathfrak{g}_1 \Bigr)
\oplus \Bigl( \mathbb{R}\langle e_1^{\prime 1},\dots,e_1^{\prime \ell_2} \rangle \oplus \mathfrak{g}_2 \Bigr),
\]
where each summand is the Lie algebra of a Joyce hypercomplex manifold associated with a simple Lie group and endowed with a strong HKT structure. The HKT-Einstein condition hence reduces to the HKT-Einstein condition on each factor. Applying Theorem \ref{thm:StrongHKTEinstein}, we get the claimed classification.
\end{proof}

\subsection{HKT-Einstein metrics on the standard Hopf surface} \label{s:Hopf}

Let $M^4=S^1\times \SU(2)$ be the Hopf surface and let
$\bar g=g_{\mathrm{Hopf}}$ be the standard locally
conformally flat metric whose HKT--Einstein constant is one.  In this section we construct an infinite sequence of HKT-Einstein metrics in this conformal class which blows up in $C^0$ norm.

We achieve this in a symmetric ansatz.  In particular we impose $g=e^u\bar g$, where $u$ is invariant under
the $S^1$-factor and under the standard $\SO(3)$-action
on $\SU(2)\simeq S^3$.  By a standard computation, for such an $u$ to define a solution of (\ref{eqn:hkteinstein4d}), we require
\begin{equation}
 u''+2\cot r\,u'=4(1-e^u).                                      \label{eq:radial}
\end{equation}
Here $r\in[0,\pi]$ is the standard polar coordinate on the $\SU(2)$-factor.
A radial solution is smooth at a pole precisely when it is even in the
geodesic distance from that pole.

We will construct the metrics to have a further $\mathbb Z_2$ symmetry given by reflection over the equator.  In particular we solve the problem on the northern hemisphere with a Neumann
condition at the equator. Indeed, if we construct a solution to (\ref{eq:radial}) on $[0, \pi/2]$ with furthermore
\begin{equation}
 u'(0)=u'\!\left(\frac\pi2\right)=0,
 \label{eq:halfBC}
\end{equation}
then the reflection
\[
 \widetilde u(r)=
 \begin{cases}
 u(r),&0\le r\le \pi/2,\\
 u(\pi-r),&\pi/2\le r\le \pi,
 \end{cases}
\]
is smooth and solves \eqref{eq:radial} on all of $[0,\pi]$, with the correct boundary conditions at both $r = 0$ and $r= \pi$.

We first construct solutions with certain boundary data at $r = 0$, and show that they always exist at least up to time $\pi/2$:
\begin{lem} \label{l:ODEexists} Given $a \in \mathbb R$ there exists a unique solution to (\ref{eq:radial}) which is regular at $r = 0$ and satisfies
\[
 u_a(0)=a,\qquad u_a'(0)=0.
\]
\end{lem}
\begin{proof} We first construct the solution to boundary value problem.  As the coefficient $\cot(r)$ is singular at the origin, we first construct solutions starting from $r = \epsilon_i \to 0$ and show that they limit to the required solution.  For fixed $i$ construct $u_a^i(r)$ the unique solution to (\ref{eq:radial}) with initial data
\begin{align*}
    u^i_a(\epsilon_i) = a, \qquad (u^i_a)'(\epsilon_i) = 0.
\end{align*}
To show these solutions converge we exploit a Lyapunov functional which we will use throughout this section.  Let
\begin{align*}
    L(v) := \frac12(v')^2+4(e^{v}-v-1).
\end{align*}
A computation shows that for $u$ a solution of (\ref{eq:radial}),
\begin{equation}
 L'=-2\cot r\,(u')^2\leq0,
 \qquad 0<r\leq\frac\pi2.                \label{eq:shooting-energy}
\end{equation}

Using the bound on $L$, the coercivity of $e^x-x-1$, and the ODE itself, it follows that there are uniform bounds on $u_a^i$, $(u_a^i)'$ and $(u_a^i)''$ up to time $\pi/2$.  By a standard compactness argument we obtain the required solution, which exists on $[0, \pi/2]$ as claimed.
\end{proof}

In view of Lemma \ref{l:ODEexists}, we may now define the shooting map
\[
 \Psi(a):=u_a'\left(\frac\pi2\right).
\]
By construction this is a continuous function of $a$, and our goal is to construct infinitely many values of $a$ for which $\Psi(a) = 0$.  This will follow as a consequence of a precise asymptotic for $\Psi$:

\begin{prop} \label{p:shootingasymptotic}
There exist constants $C>0$ and $\delta\in\R$ such that
\begin{equation}
 \Psi(a)=Ce^{-a/4}
 \cos\left(\frac{\sqrt7}{4}a+\delta\right)+o(e^{-a/4})
 \qquad(a\to+\infty).                                           \label{eq:shooting-asymptotic}
\end{equation}
\end{prop}

We will build up to the proposition through a series of lemmas.
The first step is a change of variables followed by the comparison to a universal singular solution which we expect to approximate our solutions for large $a$.  Set
\[
 t=\log\tan\frac r2,\qquad -\infty<t\leq0.
\]
The function
\[
 F(r):=-2\log(\sin r)-\log2
\]
is an exact singular solution of \eqref{eq:radial}.  Set
\[
 A:=\frac a2+\frac32\log2
\]
and, writing $r(t)=2\arctan(e^t)$, define
\[
 y_A(t):=u_a(r(t))-F(r(t)),\qquad Y_A(s):=y_A(s-A).
\]
A direct calculation gives
\begin{align}
 y_A''-\tanh t\,y_A'+2(e^{y_A}-1)&=0,                            \label{eq:cylindrical}\\
 Y_A''+Y_A'+2(e^{Y_A}-1)&=q_A(s)Y_A',                            \label{eq:shifted} \qquad q_A(s):=1+\tanh(s-A).
\end{align}
By construction, $F_t(0)=0$ and $dr/dt=1$ at $t=0$, so
\begin{equation}
 \Psi(a)=y_A'(0)=Y_A'(A).                                       \label{eq:shooting-identity}
\end{equation}
The choice of $A$ cancels the constant term in the incoming expansion.  In other words, 
for every fixed $A$,
\begin{equation}
 Y_A(s)=2s+O_A(e^{2s}),\qquad
 Y_A'(s)=2+O_A(e^{2s})\quad(s\to-\infty).                      \label{eq:incoming-YL}
\end{equation}
The next step is to show a quantitative convergence of $Y_A$ to a function $Y_{\infty}$ solving the limiting equation
\begin{equation}
 Y_{\infty}''+Y_{\infty}'+2(e^{Y_{\infty}}-1)=0. \label{eq:limiting-orbit}
\end{equation}

\begin{lem} \label{l:limiteqn} There exists a unique solution $Y_{\infty}(s)$ to (\ref{eq:limiting-orbit}) defined on $\mathbb R$ which satisfies
\begin{equation}
 Y_{\infty}(s)=2s+O(e^{2s}),\qquad Y'_{\infty}(s)=2+O(e^{2s})\quad(s\to-\infty). \label{eq:incoming-Y}
\end{equation}
Furthermore, there exist $s_0<0$, $A_0>0$, and a constant $C_{s_0}$
such that, for every $A\geq A_0$,
\begin{equation}
 |Y_A(s_0)-Y_\infty(s_0)|+|Y_A'(s_0)-Y_\infty'(s_0)|
 \leq C_{s_0}e^{-2A}.
 \label{eq:incoming-comparison}
\end{equation}
and, uniformly for $A\geq A_0$, one has
\begin{equation}
 Y_A(s)=2s+O(e^{2s}),\qquad
 Y_A'(s)=2+O(e^{2s})\quad(s\to-\infty).                      \label{eq:uniformincoming-YL}
\end{equation}
\end{lem}

\begin{proof} We will solve for functions $h_A$ obtained by removing the leading order asymptotic for $Y_A$ and define $h_A$ via the equation 
\[
 Y_A=2s+h_A.
\]
Our goal is to reconstruct $h_A$ via a Banach space fixed point argument for a certain operator $\mathcal F_A$ for large $A$.  We will furthermore show continuity of the operators $\mathcal F_A$ to a formal limiting operator $\mathcal F_{\infty}$ whose fixed point is a function $h_{\infty}$ such that $Y_{\infty} = 2s + h_{\infty}$ solves (\ref{eq:limiting-orbit}).  

We fix $s_0 < 0$ sufficiently negative and then on $(-\infty,s_0]$, $h_A$ is a fixed point of
\begin{equation}
 (\mathcal F_ A v)(s)=\int_{-\infty}^s(1-e^{-(s-\sigma)})
 \left[q_A(\sigma)(2+v'(\sigma))-2e^{2\sigma+v(\sigma)}\right]\dd\sigma.
                                                                    \label{eq:volterra}
\end{equation}
We define the operator $\mathcal F_{\infty}$ by equation (\ref{eq:volterra}) with $q_{\infty} = 0$.
Define the weighted norm
\[
 \|v\|_{s_0}:=\sup_{s\leq s_0}e^{-2s}(|v(s)|+|v'(s)|),
\]
We next claim that $\mathcal F_A$ are contractions uniformly for all sufficiently large $A$.  Fix $R>2e+1$ and assume $\|v\|_{s_0}\leq R$.  After decreasing $s_0$,
$|v(s)|\leq1, |v'(s)| \leq 1$ for $s\leq s_0$.  Therefore
\begin{align} \label{f:operatorestimate}
 e^{-2s}\bigl(|\mathcal F_Av(s)|
 +|(\mathcal F_Av)'(s)|\bigr)
 &\leq
 3e^{-2s}\int_{-\infty}^s q_A(\sigma)\dd\sigma
 +2e\,e^{-2s}\int_{-\infty}^s e^{2\sigma}\dd\sigma.
\end{align}
To estimate the right hand side we first observe
\[
 0\leq q_A(s)\leq2e^{-2A}e^{2s}.
\]
Hence the $q_A$-term in \eqref{f:operatorestimate}, together with its derivative, has weighted
norm at most $C e^{-2A}$.  The final term in (\ref{f:operatorestimate}) is easily estimated thus yielding for sufficiently large $A$ that $\mathcal F_A$ preserves the ball of radius $R$.  Using the same ideas and the mean value theorem we furthermore get
\[
 \|\mathcal F_A v-\mathcal F_A w\|_{s_0}
 \leq e e^{2s_0}\|v-w\|_{s_0}.
\]
Thus, after decreasing $s_0$ and increasing $A_0$ if necessary, the
maps $\mathcal F_A$, for $A\geq A_0$, together with
$\mathcal F_\infty$, preserve a fixed ball and have a common
contraction constant $\kappa<1$. In particular we obtain a fixed point $h_{\infty}$ of $\mathcal F_{\infty}$ in this ball.  Furthermore, the estimate \eqref{eq:incoming-YL} shows that the
geometric solution $h_A$ belongs to ball of radius $R$ in the norm defined on $(-\infty, s_A]$ for some $s_A$ depending on $A$.  Unique continuation then identifies it with the fixed point on $(-\infty,s_0]$.

We finally are ready to show that $h_A$ converges to $h_{\infty}$.  Using the uniform contraction mapping estimate as well as the estimate for the $q_A$ integral discussed above, 
\[
 \|h_A-h_{\infty}\|_{s_0}= ||\mathcal F_A h_A - \mathcal F_{\infty} h_{\infty} ||_{s_0} = ||\mathcal F_A h_A - \mathcal F_A h_{\infty} + \mathcal F_A h_{\infty} - \mathcal F_{\infty} h_{\infty} ||_{s_0} \leq \kappa\|h_A-h_{\infty}\|_{s_0}+Ce^{-2A}.
\]
Since $\kappa < 1$ this implies \eqref{eq:incoming-comparison}.
This also proves the uniform incoming control asserted after
\eqref{eq:incoming-YL}.

We finally show that $Y_{\infty}$ extends to all positive times.  Similar to Lemma \ref{l:ODEexists}, the Lyapunov functional 
\[
 E(s):=\frac12(Y_\infty')^2+2(e^{Y_\infty}-Y_\infty-1)
\]
satisfies $E'=-(Y_\infty')^2$, and the long time existence follows.
\end{proof}

\begin{lem} \label{l:Ylimitasymptotic}
Let
\[
 \lambda:=-\frac12+ i\frac{\sqrt{7}}{2}.
\]
There exists $\mu \in\mathbb C\setminus\{0\}$ such that
\begin{align}
 Y_\infty(s)&=\operatorname{Re}(\mu e^{\lambda s})+O(e^{-s}),              \label{eq:Y-asymptotic}\\
 Y_\infty'(s)&=\operatorname{Re}(\lambda \mu e^{\lambda s})+O(e^{-s})     \label{eq:Yprime-asymptotic}
\end{align}
as $s\to+\infty$.  
\end{lem}
\begin{proof}
We rewrite the equation for $Y_{\infty}$ in terms of a linear operator and an error term:
\[
 Y_\infty''+Y_\infty'+2Y_\infty=N(Y_\infty),\qquad N(Y_\infty):=-2(e^{Y_\infty}-1-Y_\infty)=O(Y_\infty^2).
\]
We first show a decay estimate for $Y_\infty$.  Let put $X=(Y_\infty,Y_\infty')$, interpreted as a column vector, and let
\[
B=\begin{pmatrix}0&1\\-2&-1\end{pmatrix}.\]
The ODE system is thus expressed as
\begin{align*}
    X'=BX+(0,N(Y_\infty)).
\end{align*}
To estimate this system we introduce the positive definite matrix
\[
P=\begin{pmatrix}7/4&1/4\\1/4&3/4\end{pmatrix},
\]
defined uniquely by the equation
\[
 B^TP+PB=-I.
\]
Using these properties, it follows easily that for $|X|$ sufficiently small, the quantity $Q=X^TPX$ satisfies
\[
 Q'=-|X|^2+2X^TP\binom{0}{N(Y_\infty)}
 \leq-\frac12|X|^2 \leq - \gamma_0 Q,
\]
where last inequality follows since $P$ is positive definite.  Thus  $|Y_\infty|+|Y_\infty'|=O(e^{-\gamma_0s})$ for some $\gamma_0>0$.

The Green kernel of $\tfrac{\del^2}{\del t^2} + \tfrac{\del}{\del t}+2$ is
\[
 K_0(t)=\frac{2}{\sqrt{7}} e^{-t/2}\sin \left(\frac{\sqrt{7}}{2} t \right),\qquad t\geq0,
\]
Consequently, if
$|Y_\infty|+|Y_\infty'|=O(e^{-\gamma s})$, Duhamel's formula gives
\[
 |Y_\infty(s)|+|Y_\infty'(s)|
 =O(e^{-s/2})+O(e^{-\min\{2\gamma,1/2\}s}),
\]
Iterating this estimate
finitely many times gives
\begin{equation}
 |Y_\infty(s)|+|Y_\infty'(s)|=O(e^{-\gamma s})
 \quad\text{for every }\gamma<\frac12.                          \label{eq:Y-decay}
\end{equation}
Now set $w=Y_\infty'-\overline\lambda Y_\infty$.  Then
\[
 w'-\lambda w=N(Y_\infty).
\]
Choosing $\gamma>1/4$ in \eqref{eq:Y-decay}, we see that
$e^{-\lambda s}N(Y_\infty(s))$ is integrable.  Therefore
\[
 c:=\lim_{s\to\infty}e^{-\lambda s}w(s)
\]
exists and
\begin{equation}
 w(s)=ce^{\lambda s}
 -e^{\lambda s}\int_s^\infty e^{-\lambda\sigma}N(Y_\infty(\sigma))\dd\sigma.
                                                                    \label{eq:w-integral}
\end{equation}
This first gives $Y_\infty,Y_\infty'=O(e^{-s/2})$, so $N(Y_\infty)=O(e^{-s})$, thus plugging back into
\eqref{eq:w-integral} gives $w=ce^{\lambda s}+O(e^{-s})$.  Taking the imaginary part of $w$ gives 
\eqref{eq:Y-asymptotic}--\eqref{eq:Yprime-asymptotic} follow with
$\mu =-\frac{2i}{\sqrt7}\,c$.

Finally we claim that $c$ cannot vanish.  If $c=0$,
\eqref{eq:w-integral} and its complex
conjugate imply, for all sufficiently large $\sigma_0$,
\[
 m_{\sigma_0}\leq Ce^{-\sigma_0/2}m_{\sigma_0}^2,
 \qquad
 m_{\sigma_0}:=\sup_{\sigma\geq\sigma_0}e^{\sigma/2}|Y_\infty(\sigma)|.
\]
Noting that $m_{\sigma_0}$ is uniformly bounded by the proven decay estimates, this is a contradiction for sufficiently large $\sigma_0$ unless $m_{\sigma_0} = 0$.  However if $m_{\sigma_0} = 0$ then $Y_{\infty}$ vanishes on a forward half-line, and uniqueness in
\eqref{eq:limiting-orbit} would give $Y_{\infty}\equiv0$, contradicting
\eqref{eq:incoming-Y}.  Thus $c\neq0$ and thus $\mu \neq 0$.
\end{proof}

\begin{proof}[Proof of Proposition \ref{p:shootingasymptotic}] The proof builds on Lemma \ref{l:limiteqn}, which shows that $Y_A$ is well-approximated by $Y_{\infty}$ at time close to $-\infty$, and Lemma \ref{l:Ylimitasymptotic}, which gives the long time behavior of the limit solution $Y_{\infty}$.  We now propagate the estimates forward in time to show that $Y_A$ is well-approximated by $Y_{\infty}$ at $s = A$.

We fix small $\epsilon > 0$ and let $s_A:=\epsilon A$.  Writing \eqref{eq:shifted} and \eqref{eq:limiting-orbit} as first-order
systems and using Gronwall's inequality on the interval $[s_0, s_A]$ and the estimate \eqref{eq:incoming-comparison}, give
\begin{align*}
 \left|\binom{Y_A(s_A)}{Y_A'(s_A)}-
 \binom{Y_\infty(s_A)}{Y_\infty'(s_A)}\right|
 &\leq Ce^{K (s_A - s_0)}\left(e^{-2A}+
 \int_{s_0}^{s_A}q_A(\sigma)\dd\sigma\right)\\
 &\leq C_{s_0}\exp\bigl([K\epsilon-2(1-\epsilon)]A\bigr),
\end{align*}
where $K$ is independent of $A$ and we used
\[
 \int_{s_0}^{s_A}q_A(\sigma)\dd\sigma \leq \int_{-\infty}^{s_A}q_A(\sigma)\dd\sigma
 =\log(1+e^{2(s_A-A)})\leq e^{-2(1-\epsilon)A}.
\]
Choose $\epsilon$ so small that
\[
 K\epsilon-2(1-\epsilon)<-\frac{\epsilon}{2}.
\]
Equations \eqref{eq:Y-asymptotic}--\eqref{eq:Yprime-asymptotic} then yield
\begin{equation}
 \binom{Y_A(s_A)}{Y_A'(s_A)}
 =\operatorname{Re}\left[\mu e^{\lambda s_A}
 \binom{1}{\lambda}\right]+o(e^{-s_A/2}).                       \label{eq:core-data}
\end{equation}

We now return to the original time variable $t = s - A$ and then define 
\[
t_A:=s_A-A=-(1-\epsilon)A.
\]
The final step is to propagate the estimate through $[t_A, 0]$, achieved by comparison to the relevant linearized equation.  Linearizing \eqref{eq:cylindrical} at zero gives
\begin{equation}
 \mathcal Tz:=z''-\tanh t\,z'+2z=0.                              \label{eq:terminal-linear}
\end{equation}
By arguments similar to those used above, one easily shows there is a unique complex solution $Z$ of \eqref{eq:terminal-linear} such that
\begin{equation}
 Z(t)=e^{\lambda t}(1+O(e^{2t})),\qquad
 Z'(t)=\lambda e^{\lambda t}(1+O(e^{2t}))                       \label{eq:terminal-mode}
\end{equation}
as $t\to-\infty$.  Moreover, if $U(t,\sigma)$ denotes the fundamental matrix of
\eqref{eq:terminal-linear} in the variables $(z,z')$, we claim
\begin{equation}
 \|U(t,\sigma)\|\leq Ce^{-(t-\sigma)/2},
 \qquad -\infty<\sigma\leq t\leq0.                              \label{eq:terminal-propagator}
\end{equation}
To prove this, we consider the change of variable 
\[
v(t):=(\cosh t)^{-1/2}z(t),
\]
and the energy functional
\[
 E(v) := \frac{1}{2}|v'|^2+
 \frac{1}{2}\left(\frac{7}{4}+\frac{3}{4}\operatorname{sech}^2t\right)|v|^2.
\]
A computation shows that
\[
\brs{E'} \leq \frac{6 }{7} \operatorname{sech}^2 t E,
\]
which yields (\ref{eq:terminal-propagator}) by integration.

Define on $[t_A,0]$ the linear approximation
\[
 z_A(t):=e^{-A/2}\operatorname{Re}\bigl(\mu \, e^{i\frac{\sqrt7}{2} A}Z(t)\bigr).
\]
Because
\[
 e^{-A/2}e^{i\frac{\sqrt7}{2} A}e^{\lambda t_A}=e^{\lambda s_A},
\]
\eqref{eq:core-data} and \eqref{eq:terminal-mode} imply
\begin{equation}
 \left|\binom{y_A(t_A)}{y_A'(t_A)}-
 \binom{z_A(t_A)}{z_A'(t_A)}\right|=o(e^{-s_A/2}).               \label{eq:terminal-data}
\end{equation}
Let $w_A:=y_A-z_A$.  From \eqref{eq:cylindrical} and
\eqref{eq:terminal-linear},
\begin{equation}
 \mathcal Tw_A=-2(e^{y_A}-1-y_A).                            \label{eq:terminal-error}
\end{equation}
We define a norm for functions $v$ on $[t_A, 0]$ via 
\[
 \|v\|_{t_A}:=\sup_{t_A\leq t\leq0}
 e^{A/2}e^{t/2}(|v(t)|+|v'(t)|).
\]
The functions $z_A$ have uniformly bounded norm, and
\eqref{eq:terminal-propagator}--\eqref{eq:terminal-data} show
that the homogeneous propagation of the initial error is $o(1)$ in this norm.
For times $T$ such that
\begin{align*}
    \sup_{t_A \leq t \leq T < 0}   
 e^{A/2}e^{t/2}(|y_A(t)|+|y_A'(t)|) \leq R,
\end{align*}
observe
\[
 |y_A(t)|\leq Re^{-A/2}e^{-t/2}\leq Re^{-s_A/2},
\]
and hence, for large $A$,
\[
 |e^{y_A}-1-y_A|\leq C|y_A|^2\leq CR^2e^{-A}e^{-t}.
\]
Duhamel's formula and \eqref{eq:terminal-propagator} therefore give
\[
 \sup_{t_A \leq t \leq T < 0}   
 e^{A/2}e^{t/2}(|y_A(t)|+|y_A'(t)|) \leq C_0+o(1)+Ce^{-s_A/2}R^2.
\]
Taking $R>2C_0$ and using a standard continuity argument yields
\begin{equation}
 |y_A(t)|+|y_A'(t)|\leq Ce^{-A/2}e^{-t/2},
 \qquad t_A\leq t\leq0.                                       \label{eq:terminal-bound}
\end{equation}
Applying Duhamel's formula to \eqref{eq:terminal-error} at $t=0$ and using
\eqref{eq:terminal-propagator}, \eqref{eq:terminal-data}, and
\eqref{eq:terminal-bound}, we obtain
\begin{align*}
 |w_A'(0)|
 &\leq Ce^{t_A/2}o(e^{-s_A/2})
 +Ce^{-A}\int_{t_A}^0e^{-t/2}\dd t\\
 &=o(e^{-A/2})+O(e^{-A/2}e^{-s_A/2})
 =o(e^{-A/2}).
\end{align*}
Consequently, by \eqref{eq:shooting-identity},
\begin{equation}
 \Psi(a)=e^{-A/2}
 \operatorname{Re}\bigl(\mu Z'(0)e^{i\frac{\sqrt7}{2} A}\bigr)+o(e^{-A/2}).  \label{eq:shooting-final}
\end{equation}
Since $A=a/2+(3/2)\log2$,
\eqref{eq:shooting-final} becomes
\[
 \Psi(a)=e^{-a/4}\operatorname{Re}
 \left(Be^{i\sqrt7a/4}\right)+o(e^{-a/4}),
\]
where
\[
 B:=2^{-3/4}\mu Z'(0)
 \exp\left(i\frac{3\sqrt7}{4}\log2\right).
\]
The final step is to show that $B \neq 0$, which will follow if we show $Z'(0) \neq 0$ is nonzero.  Since the coefficients of (\ref{eq:terminal-linear}) are real $\bar{Z}$ is also a solution, and it follows that the Wronskian $W = Z \bar{Z}' - Z' \bar{Z}$ is a solution of $W' = \tanh t W$ which is moreover nowhere vanishing.  It follows that $Z'(0) \neq 0$.  Expressing $B=Ce^{i\delta}$ thus proves \eqref{eq:shooting-asymptotic}.
\end{proof}

\begin{thm} \label{t:4dnonuniquenessbulk} 
Consider the standard quaternionic Hopf surface $(S^1 \times \SU(2), I,J,K, g_{\mathrm{Hopf}})$. There exists a sequence $\{u_k\} \in C^{\infty}(S^1 \times \SU(2))$ such that
\begin{enumerate} [label={(\arabic*)}]
    \item $\Phi(e^{u_k} g_{\mathrm{Hopf}}) = e^{u_k} g_{\mathrm{Hopf}}$,
    \item $\lim_{k \to \infty} \brs{u_k}_{C^0} = \infty$.
\end{enumerate}
\end{thm}

\begin{proof} Following the prior discussion, it suffices to find infinitely many $a_k \in \mathbb R$ such that $\Psi(a_k) = 0$.  By Proposition~\ref{p:shootingasymptotic}, there exist constants
$C>0$ and $\delta\in\mathbb R$ such that
\[
 \Psi(a)
 =
 Ce^{-a/4}\cos\left(\frac{\sqrt7}{4}a+\delta\right)
 +o(e^{-a/4})
\]
as $a\to+\infty$. We set
\[
 b_k:=
 \frac{4}{\sqrt7}
 \left(\left(k+\frac12\right)\pi-\delta\right),
\]
and an elementary argument shows that there exists small $\epsilon_k > 0$ so that
\begin{align*}
    \Psi(b_k- \epsilon_k)  \Psi(b_k + \epsilon_k) < 0,
\end{align*}
Hence $\Psi$ has a zero $a_k$ close to $b_k$ by continuity. 
By the construction above, each zero $a_k$ determines a function
$u_k\in C^\infty(S^1\times\SU(2))$ satisfying
\[
 \Phi(e^{u_k}g_{\mathrm{Hopf}})
 =
 e^{u_k}g_{\mathrm{Hopf}},
\]
and moreover since $u_k(0)=a_k$,
\[
 \brs{u_k}_{C^0}
 \geq a_k
 \longrightarrow+\infty.
\]
\end{proof}


\providecommand{\bysame}{\leavevmode\hbox to3em{\hrulefill}\thinspace}
\providecommand{\MR}{\relax\ifhmode\unskip\space\fi MR }
\providecommand{\MRhref}[2]{%
  \href{http://www.ams.org/mathscinet-getitem?mr=#1}{#2}
}
\providecommand{\href}[2]{#2}

\end{document}